\documentclass[12pt]{amsart}

\usepackage{amssymb}
\usepackage{amsmath}
\usepackage{amsthm}
\usepackage{amscd}
\usepackage[margin=1in]{geometry}
\usepackage{hyperref}

\theoremstyle{plain}
\newtheorem{theorem}{Theorem}[section]
\newtheorem{proposition}{Proposition}[section]
\newtheorem{corollary}{Corollary}[section]
\newtheorem{lemma}{Lemma}[section]
\newtheorem*{blanktheorem}{Theorem}
\newtheorem*{blankassumption}{Assumption}

\theoremstyle{remark}
\newtheorem{definition}{Definition}[section]
\newtheorem{remark}{Remark}[section]

\newtheorem{assumption}{Assumption}[section]

\DeclareMathOperator*{\esssup}{ess\,sup}
\DeclareMathOperator{\dist}{dist}
\DeclareMathOperator{\supp}{supp}

\DeclareMathOperator{\cpct}{Cap}
\DeclareMathOperator{\clos}{clos}
\DeclareMathOperator{\lin}{span}
\DeclareMathOperator{\curl}{curl}
\DeclareMathOperator{\diverg}{div}
\DeclareMathOperator{\dom}{dom}
\DeclareMathOperator{\dvol}{dvol}
\begin{document}

\title[Local Dirichlet forms, Hodge theory, and the 
NSE on    fractals]{Local Dirichlet forms, Hodge theory, and the 
Navier-Stokes equations on topologically one-dimensional  fractals}
\author[Michael Hinz]{Michael Hinz$^{1,2}$}
\address{Department of Mathematics, Universit\"at Bielefeld, Postfach 100131, D-33501 Bielefeld, Germany}
\email{mhinz@math.uni-bielefeld.de}
\thanks{$^1$Research supported in part by  the Alexander von Humboldt Foundation (Feodor Lynen Research Fellowship Program) and carried out during a stay at the Department of Mathematics, University of Connecticut, Storrs, CT 06269-3009 USA}
\author[Alexander Teplyaev]{Alexander Teplyaev$^2$}
\address{Department of Mathematics, University of Connecticut, Storrs, CT 06269-3009 USA}
\email{Alexander.Teplyaev@uconn.edu}
\thanks{$^2$Research supported in part by NSF grant DMS-0505622}

\date{\today}

\begin{abstract}

We consider finite energy and $L^2$ differential forms associated with strongly local regular Dirichlet forms on compact connected topologically one-dimensional spaces. We  introduce notions of local exactness and local harmonicity and prove the Hodge decomposition, which in our context says that the orthogonal complement to the space of all exact 1-forms coincides with the closed span of all locally harmonic 1-forms.  Then we introduce a related Hodge Laplacian and define a notion harmonicity for finite energy 1-forms. As as corollary, under a certain capacity-separation assumption, we prove that the space of harmonic 1-forms is nontrivial if and only if the classical  \v{C}ech cohomology is nontrivial. In the examples of classical self-similar fractals these spaces typically are either trivial or infinitely dimensional. Finally, we study Navier-Stokes type models and prove that under our assumptions they have only steady state divergence-free solutions. In particular, we solve the existence and uniqueness problem for the Navier-Stokes and Euler equations for a large class of fractals that are topologically one-dimensional but can have arbitrary Hausdorff and spectral dimensions. 

\tableofcontents
\end{abstract}
\subjclass[2000]{Primary 
31E05 
60J45; 
Secondary 
28A80
31C25
35J25 
35Q30 
46L87 
47F05 
58J65
60J60 
81Q35}
\maketitle

\section{Introduction}
In this article we study the finite energy de Rham complex on compact connected topologically one-dimensional spaces $X$ which carry a strongly local regular Dirichlet (energy) form $(\mathcal{E},\mathcal{F})$. There are large classes of fractals of arbitrary Hausdorff dimension $1\leq d_H <\infty$ and various possible spectral dimensions that have these properties (see \cite[and references therein]{Ba,BBKT,Ki01,Ki12,Li,RT10,S-POTA,ST-BE,T08} and Remark~\ref{rem-ds}). 
During the recent years Dirichlet forms have proven to be efficient tools in various areas of analysis and probability. Some recent papers related to our research that use Dirichlet form methods are for instance \cite{BBK06,BGK11,B11}. 
We investigate the de Rham complex consisting of finite energy $0$-forms and square integrable $1$-forms that results from constructions introduced in  \cite{CS03,CS09,IRT}. We introduce new localization techniques and prove that locally exact forms are dense in the space of $1$-forms, which allows to define a Hodge Laplacian on  $1$-forms.

One of the aims of our article is to generalize results previously obtained in \cite{CGIS11} and \cite{IRT}, which connect $1$-forms and topology (thus we relate differential geometry on fractals and topology). These papers contain, among several other results, Hodge type decompositions for the Sierpinski gasket with its standard Dirichlet form, \cite{CGIS11}, and more generally, for finitely ramified fractals carrying a resistance form, \cite{IRT}. 
It was shown in \cite{IRT}, in the finitely ramified case, that the Hilbert space of harmonic forms is trivial if and only if the space is a tree. Moreover, harmonic forms 'count' the cycles of these topologically one-dimensional spaces. In contrast to classical smooth examples (e.g. \cite{BT82, Jost, W71}), the space of harmonic forms on self-similar fractals will often be either trivial or infinite dimensional. 

Being a Hilbert space, the space of $1$-forms is self-dual, which allows to identify $1$-forms and vector fields. Doing so allows us to introduce some elements of vector analysis, as recently done in \cite{HRT} (which generalizes earlier approaches to vector analysis on fractals, see \cite{Ki93h,Ki08,Ku93,St00t,T08}). This is a part of a more comprehensive program addressing vector equations on general, possibly fractal, spaces which carry a diffusion, a topic that is still brand new. Some related scalar PDE are discussed in \cite{HRT} and first applications to magnetic fields are considered in \cite{HT12a,HT12}. 

In this paper  we introduce hydrodynamical models on fractals, such as   Navier-Stokes type systems (see e.g. \cite{AMR, ChM, Te83} for background). In a sense, this justifies that our notion of finite energy differential forms is mathematically natural and implies physically relevant properties of classical differential equations on fractals. Note, however, that on classical fractals (e.g. the Sierpinski gasket) there are no continuous tangent vector fields in the classical sense (see \cite{T04} for explanations) and so our notions correspond to the so called measurable Riemannian geometry (a term introduced by Kigami in \cite{Ki08}).

Overall our results revolve around the the notions of local exactness and local harmonicity of differential $1$-forms. As parts of the paper are quite technical, we take a moment to outline some ideas, definitions and main results. Consider a smooth compact one-dimensional manifold $M$ without boundary. Of course there is not much variety, as any such $M$ must be diffeomorphic to the circle. The de Rham 
complex of $M$ has the structure
\begin{equation}\label{E:deRham}
0\xrightarrow{} \Omega^0(M) \xrightarrow{d_0} \Omega^1(M)\xrightarrow{d_1} 0,
\end{equation}
where $\Omega^0(M)=C^\infty(M)$ and $\Omega^1(M)$ are the spaces of smooth functions and $1$-forms on $M$, respectively, and the $d_i$ denote the respective instances of the exterior derivative. If, although not needed to study de Rham theory, the manifold $M$ is endowed with a Riemannian structure, then we can consider the Dirichlet integral
\begin{equation}\label{E:Em}
\mathcal{E}_M(f)=\int_M |df|^2\dvol, \ \ f\in C^\infty(M),
\end{equation}
where $d=d_0$ and $\dvol$ denotes the Riemannian volume. Taking its closure yields a symmetric local regular Dirichlet $(\mathcal{E}_M,H^1(M))$ form on $L_2(M,\dvol)$, where $H^1(M)$ is the corresponding Sobolev space of functions on $M$. In the present paper we investigate general topologically one-dimensional compact metric spaces $X$ that carry a finite Radon measure $m$ with full support and a symmetric strongly local regular Dirichlet form $(\mathcal{E},\mathcal{F})$ on $L_2(X,m)$. We assume that it admits a spectral gap and, to exclude potential theoretic complications, that the associated transition kernels are absolutely continuous with respect to the reference measure $m$. We will frequently make use of the related energy measures $\Gamma(f)$, which can be defined for any energy finite function $f\in\mathcal{F}$ and in particular satisfy 
\[\mathcal{E}(f)=\int_Xd\Gamma(f).\]
More details can be found in Section \ref{S:Setup} or in \cite{CF,FOT, LeJan78}. For the Dirichlet form (\ref{E:Em}) for instance we observe $\Gamma(f)=|df|^2\dvol$, $f\in H^1(M)$. 

What concerns differential forms on $X$, the first question would be how to define them. We follow the approach of Cipriani and Sauvageot \cite{CS03, CS09} that, in some sense, reverses the order of definition of the objects in (\ref{E:deRham}) and (\ref{E:Em}). There the authors start with a Dirichlet form and construct a Hilbert space $\mathcal{H}$ as the closure of tensor products of bounded, energy finite functions (in a certain anti-symmetrizing norm, which may be expressed using the energy measures induced by the Dirichlet form). They refer to it as \emph{the space of $1$-forms on $X$}, and so will we. The space of ($m$-a.e.) bounded energy finite functions $\mathcal{B}:=\mathcal{F}\cap L_\infty(X,m)$ forms an algebra. A related derivation operator $\partial:\mathcal{B}\to \mathcal{H}$ takes such functions into $1$-forms, and it may be extended to a unbounded closed linear operator 
on $L_2(X,m)$ with domain $\mathcal{F}$. A special case of this construction yields Weaver's derivation for $L_\infty$-diffusions, \cite{W00}. In cases where both theories coexist, $\partial f$ coincides with the upper gradient of $f$ in the sense of Cheeger, see \cite{Ch99} and also \cite{CS03, Hei01, HeiK98, KZ12, W00}. The norm in $\mathcal{H}$ is such that for $1$-forms of type $g\partial f$ we have
\[\left\|g\partial f\right\|_\mathcal{H}^2=\int_X g^2d\Gamma(f),\]
and the collection of such $1$-forms $g\partial f$ is dense in $\mathcal{H}$. This approach is an $L_2$-formulation, and in general differential $1$-forms $\omega\in\mathcal{H}$ cannot be evaluated pointwise. Note that for the Dirichlet form as in (\ref{E:Em}) we have in particular
\[\left\|g\partial f\right\|_\mathcal{H}^2=\int|gdf|^2\dvol, \ \ f,g\in C^\infty(M),\]
extending (\ref{E:Em}). Therefore $\mathcal{H}$ is a generalization of the space $L_2(M,\dvol, \Lambda^1TM)$ of square integrable differential $1$-forms on $M$, and the operator $\partial$ generalizes the closure of $d=d_0: C^\infty(M)\to L_2(M,\dvol, \Lambda^1TM)$ in $L_2(M,\dvol)$. If in (\ref{E:deRham}) we replace the spaces $\Omega^i(M)$ by the spaces of square integrable functions and $1$-forms and consider $L_2$-derivations $\partial_i$ instead of the $d_i$, we get an $L_2$-differential complex. Its generalization in our setup should read
\begin{equation}\label{E:complex}
0\xrightarrow{} L_2(X,m)\xrightarrow{\partial_0}\mathcal{H}\xrightarrow{\partial_1}0
\end{equation}
with $\partial_0=\partial$, which is an unbounded closed densely defined linear operator 
on $L_2(X,m)$ with domain $\mathcal{F}$. However, a concept of $2$-forms and a derivation $\partial_1$ taking $1$-forms into $2$-forms is not yet fully established and therefore writing (\ref{E:complex}) still involves the \emph{assumption} that the space of $2$-forms is trivial. Wherever we state definitions motivated by this assumption we will mention it explicitly and comment on it. 

A key property of the de Rham complex (in fact, of any differential complex) is $d^2=0$, i.e. all exact $1$-forms $\omega=df$ are closed, $d\omega=0$. In the one-dimensional case of course all $1$-forms are closed. On the other hand every closed $1$-form in the de Rham sense is locally exact by a special case of the Poincar\'e lemma. In \cite{CGIS11} it has been shown that this implication cannot be expected to hold in our context. More precisely, the authors have constructed a non-locally exact $1$-form on the Sierpinski gasket, which is topologically one-dimensional and fits into our framework. Here we define local exactness as follows, cf. Definition \ref{D:locexact}: A $1$-form $\omega\in\mathcal{H}$ is called \emph{locally exact} if there exist a finite open cover $\mathcal{U}=\left\lbrace U_\alpha\right\rbrace_{\alpha\in J}$ of $X$ and functions $f_\alpha\in\mathcal{B}$, $\alpha\in J$, such that
\[\omega\mathbf{1}_{U_\alpha}=\partial f_\alpha\mathbf{1}_{U_\alpha}, \ \alpha\in J .\]

We formulate and prove our results under the following standing assumption,
stated as Assumption \ref{A:assumption} below. Given a closed set $F\subset X$ we denote by $\mathcal{S}^F$ the collection of all functions $f\in \mathcal{F}\cap L_\infty(X,m)$ for which there exists a finite open cover of $F$ such that the quasi-continuous version $\widetilde{f}$ of $f$ is q.e. constant on any connected component of this cover.

\begin{blankassumption}
There is a topological base $\mathcal{O}$, stable under taking finite unions of sets, such that for any $V\in\mathcal{O}$ and any $f\in\mathcal{F}$ there is a sequence of functions $(f_n)_n\subset \mathcal{S}^{\partial V}$ such that $\lim_n\mathcal{E}(f-f_n)=0$. 
\end{blankassumption}
This assumption can for instance be verified if the boundaries $\partial V$ have a certain Cantor set structure and the domains of the trace Dirichlet forms on the $\partial V$ have dense subspaces consisting of H\"older continuous functions. For a large class of self-similar Sierpinski carpets this can be concluded using \cite[Remark 2.7 and Remark 3.10]{HinoKumagai06}.

Then Theorem \ref{T:Hcoincide} (i) reads as follows.

\begin{blanktheorem}
The collection of locally exact $1$-forms is dense in $\mathcal{H}$.
\end{blanktheorem}

This is a strong indication that we should assume the space of $2$-forms to be trivial, as in (\ref{E:complex}): any useful definition of $\partial_1$ should be local. Therefore $\partial_1$ should vanish on locally exact forms, i.e. they should be closed. As they are dense in $\mathcal{H}$ it is not too far fetched to expect $\mathcal{H}=ker\:\partial_1$. 

Next, recall that in the classical one-dimensional case we have the $L_2$-Hodge decomposition
\[L_2(M,\dvol,\Lambda^1 TM)=Im\:\partial_0\oplus \mathbb{H},\]
where $Im\:\partial_0$ is the image of $\partial_0$ and $\mathbb{H}$ is the space of $\omega\in L_2(M,\dvol,\Lambda^1 TM)$ that satisfy
$0=\partial_1\omega=\partial_0^ \ast\omega$. Here $\partial^\ast_0$ denotes the $0$-codifferential of $\partial_0$, that is the formal adjoint of $\partial_0$. The space $\mathbb{H}$ is called the space of harmonic $1$-forms. See for instance \cite{GMS}. Also in our setup $Im\:\partial$ is a closed subspace of $\mathcal{H}$ and therefore 
\begin{equation}\label{E:hodgedecompintro}
\mathcal{H}=Im\:\partial \oplus \mathcal{H}^1(X),
\end{equation}
where $\mathcal{H}^1(X)=(Im\:\partial)^\bot$ is the orthogonal complement of $Im\:\partial$. For the example (\ref{E:Em}) we clearly have $\mathcal{H}^1(M)=\mathbb{H}$. We may ask about the significance of the space $\mathcal{H}^1(X)$ in the general case. 

In classical theory central results of de Rham theory tell that the first de Rham cohomology 
\[H_{dR}^1(M):=ker\:d_1/Im\:d_0\]
of a smooth compact manifold $M$ is finite dimensional and isomorphic to the first \v{C}ech cohomology $\check{H}^1(M)$ of $M$, see e.g. \cite{BT82, Jost, W71}. Moreover, the Hodge theorem tells that in every cohomology class $[\omega]\in H_{dR}^1(M)$ there exist a unique member of $\mathbb{H}$. Also $\mathbb{H}$ is finite dimensional, as follows from compactness and ellipticity. In \cite{IRT} it was shown that for finitely ramified fractals $X$ carrying a resistance form the space $\mathcal{H}^1(X)$ is trivial if and only if $X$ is a tree. Moreover, its elements were seen to 'count' the cycles of graph approximations of these topologically one-dimensional spaces. For p.c.f. self-similar fractals for instance the space $\mathcal{H}^1(X)$ is therefore typically either trivial or infinite dimensional. For combinatorially finite metric graphs it may be finite dimensional, \cite{IRT}, and trivial if and only if the graph is a tree. Note that our theory is very different from the smooth situation, because it allows complicated gluing of circles. From a topological point of view the classical Sierpinski gasket for instance arises as a limit of spaces obtained by gluing circles with 'contradicting' orientations. 

Recall that in the $L_2$-context cohomology groups should generally not be expected to allow simple statements on topology, as their dimensions are no topological invariants. What we prove here in our setup is the following qualitative result: If $\mathcal{H}^1(X)$ is nontrivial, then also the first \v{C}ech cohomology $\check{H}^1(X)$ of $X$ must be nontrivial, and under some potential theoretic condition a nontrivial first cohomology $\check{H}^1(X)$ also guarantees the nontriviality of $\mathcal{H}^1(X)$. To establish this result we introduce a notion of local harmonicity: A $1$-form $\omega\in\mathcal{H}$ is called \emph{locally harmonic} if there exist a finite open cover $\mathcal{U}=\left\lbrace U_\alpha\right\rbrace_{\alpha\in J}$ of $X$ and functions $h_\alpha\in\mathcal{B}$, $\alpha\in J$, such that each $h_\alpha$ is harmonic on $U_\alpha$ and
\[\omega\mathbf{1}_{U_\alpha}=\partial h_\alpha\mathbf{1}_{U_\alpha}, \ \alpha\in J .\]
Here the term 'harmonic function' is used in the Dirichlet form sense, see Section \ref{S:Setup}. Another denseness result is Theorem \ref{T:Hcoincide} (ii). 

\begin{blanktheorem}
The collection of locally harmonic $1$-forms is dense in $\mathcal{H}^1(X)$.
\end{blanktheorem}

Finally, we define the notion of harmonicity for $1$-forms. To do so, we introduce a Hodge Laplacian $\Delta_1$ on $1$-forms by setting
\begin{equation}\label{E:hodgelaplaceintro}
\Delta_1:=\partial\partial^\ast
\end{equation}
and observe the following, cf. Theorem \ref{T:hodgesa}.

\begin{blanktheorem}
The Hodge Laplacian $\Delta_1$ defines a self-adjoint operator on $\mathcal{H}$.
\end{blanktheorem}

Of course definition (\ref{E:hodgelaplaceintro}) is made with the expectation that the space of $2$-forms will be trivial, i.e. that any $1$-form on $X$ will be closed. In view of our previous results this seems most reasonable. We agree to say that a $1$-form $\omega\in\mathcal{H}$ is \emph{harmonic} if it is in the kernel of $\Delta_1$. We characterize harmonicity in Theorem \ref{T:critharmonic}. Combined with (\ref{E:hodgedecompintro}) it leads to an analog of the Hodge decomposition theorem and tells that $\mathcal{H}^1(X)$ should be called the \emph{space of harmonic $1$-forms}.

\begin{blanktheorem}
A $1$-form $\omega\in\mathcal{H}$ is harmonic if and only if $\omega\in\mathcal{H}^1(X)$. Consequently a $1$-form is harmonic if and only if it is orthogonal to the space of exact $1$-forms, and every $1$-form may be written as the orthogonal sum of an exact and a harmonic part. 
\end{blanktheorem}

As an application of our results we study Navier-Stokes type models on compact connected topologically one-dimensional spaces. On the circle the Navier-Stokes system simplifies to an Euler type equation, which has only steady state solutions. We observe a similar behavior in the present case. The main difference is that on a fractal space, one can have infinitely many nontrivial solutions corresponding to various cycles in the space. If the system is considered without boundary conditions, nontrivial solutions  exist if and only if the first \v{C}ech cohomology does not vanish. If time-independent boundary conditions are imposed, there may exist additional nontrivial solutions that are gradients of harmonic functions. Note  that we do not have to consider so-called tamed Navier-Stokes equations (see \cite{RZ} and references therein). 

\begin{remark}\label{rem-ds}

In several places our arguments crucially rely on topological one-dimensionality and compactness. The results apply to classical smooth examples (the real line, intervals and circles), quantum graphs (see \cite{BerKuchBook} and references therein), as well as to fractals such as p.c.f. self similar sets or nested fractals, \cite{Ba, HMT, IRT, Ki01, Li}, generalized Sierpinski carpets \cite{BB99, BBKT, MTW} of topological dimension one, Barlow-Evans-Laakso spaces and their generalizations \cite{BE04, La00, Stei10, S-POTA, ST-BE},  diamond fractals \cite{KSW} and some random fractals \cite{H92, H97}. Fractal examples are not required to be finitely ramified \cite{T08} or self-similar. 
Our research is influenced by the analysis with respect to the energy measures on singular spaces 
\cite{Hino05,Hino06,Hino12,Hino13,Kajino2012,Ki08}. 
We are especially interested in applications to analysis and geometry on metric measure spaces, see \cite[and references therein]{Ch99, Hei01, HeiK98, KZ12, W00}.

Note in particular that there are spaces of \emph{any} Hausdorff dimension $1\leq d_H<\infty$ to which our results apply, for instance, spaces of Barlow-Evans-Laakso type. Of special importance is the relation between our work and \cite{MTW}, where the authors consider generalized Sierpinski carpets of positive two-dimensional Lebesgue measure that are topologically one-dimensional. This relation will be the subject of further study. 

All this also implies that we can find examples for a broad range of spectral dimensions $d_S$.  Recall that by definition $d_S=2d_H/d_w$, where $d_w\geq 2$ is the so-called walk index (walk dimension). In \cite{Ba04} a related result for weighted graphs states that for any pair of numbers $(\alpha, d_w)$ with $\alpha\geq 1$ and $2\leq d_w\leq 1+\alpha$, there exist a weighted graph carrying an Ahlfors $\alpha$-regular measure and a random walk with walk index $d_w$. For prefractal graphs $d_w>2$ is typical. 
\end{remark}

This paper is organized as follows. In Section \ref{S:Setup} we start with a symmetric strongly local regular Dirichlet form $(\mathcal{E},\mathcal{F})$ and recall the definition of energy measures, capacities and harmonic functions. We state a maximum principle and finally some facts about hitting kernels. Section \ref{S:subdomains} is rather technical. It introduces Dirichlet subdomains obtained by completion from functions that are locally constant on a neighborhood of a given compact set. Using suitable covers, partitions of unity and the regularity of the Dirichlet form, the original Dirichlet domain $\mathcal{F}$ is recovered as the sum of two such subdomains for disjoint compact sets, see Theorem \ref{T:reconstruct}. Following \cite{CS03, CS09, IRT} we introduce $1$-forms and derivations in Section \ref{S:locally}. We define and discuss the notions of local exactness and local harmonicity as outlined above. Note that in \cite{CGIS11} and \cite{IRT} the spaces under consideration possess rigid structures which support rather transparent proofs by graph approximations. In our paper we do not assume any specific cell structure, and turn instead to finite open covers in order to connect $1$-forms and topology. Using solely the definition of $1$-forms, it therefore seems rather difficult to describe the space $\mathcal{H}^1(X)$, but the notion of local harmonicity allows to link simple $1$-forms to finite open covers. To do so we show that a certain space $\mathcal{S}_{loc}$ of tensor products of bounded energy finite functions and indicators of open sets with zero dimensional boundaries is dense in the space $\mathcal{H}$. The space $\mathcal{S}_{loc}$ is considerably easier to handle than $\mathcal{H}$ itself, and using its denseness, we verify that the spaces of locally exact and locally harmonic forms are dense in $\mathcal{H}$ and $\mathcal{H}^1(X)$, respectively. In Section \ref{S:nontrivial} these denseness results are employed to give the mentioned topological characterization for the nontriviality of the space $\mathcal{H}^1(X)$ in terms of the nontriviality of the first \v{C}ech cohomology $\check{H}^1(X)$. The potential theoretic condition needed here is that, roughly speaking, every set disconnecting a connected open set into two disjoint open pieces has positive capacity. Sufficient conditions for the validity of this capacity condition can be given in terms of irreducibility. For precise statements see Section \ref{S:nontrivial}. In Section \ref{S:harmonicforms} we first recall some notions of vector analysis proposed in \cite{HRT}. Then we define the Hodge Laplacian (\ref{E:hodgelaplaceintro}) on $1$-forms and prove it yields a self-adjoint operator. We introduce harmonicity and observe that the space of harmonic $1$-forms equals $\mathcal{H}^1(X)$. Section \ref{S:NS} is devoted to the Navier-Stokes model. Here the main part of the necessary work lies in providing the preliminaries needed to make the model rigorous. In order to achieve this, we use local harmonicity and weighted energy measures to find a suitable substitute for the convection term. The mentioned statements on stationarity and nontriviality of solutions then follow naturally. The case of time-independent boundary conditions is elaborated on in Section \ref{S:resistance} within the context of resistance forms. 

To simplify notation, sequences or families indexed by the naturals will be written with index set suppressed, e.g. $(a_n)_n$ stands for $(a_n)_{n\in\mathbb{N}}$. Similarly, $\lim_n a_n$ abbreviates $\lim_{n\to \infty}a_n$.

\subsection*{Acknowledgements} We are grateful to Michael R\"ockner for helpful comments concerning the Navier-Stokes equations and to Naotaka Kajino for pointing out some errors in an earlier version of this paper.

\section{Setup and preliminaries}\label{S:Setup}

In this section we describe our setup in detail and briefly discuss some preliminary facts used in the sequel.

In our paper $(X,d)$ is assumed to be a connected topologically $1$-dimensional compact (hence separable) metric space. Furthermore, $m$ is assumed to be a finite Radon measure on $X$ such that $m(U)>0$ for any open set $U\subset X$. We finally assume that $(\mathcal{E},\mathcal{F})$ is a symmetric strongly local regular Dirichlet form on $L_2(X,m)$ which has a spectral gap, i.e. there exists some $c>0$ such that for all $f\in\mathcal{F}$ we have 
\[\int_X(f-f_X)^ 2dm\leq c\:\mathcal{E}(f),\]
where
\[f_X=\frac{1}{m(X)}\int_Xfdm.\]
With $\mathcal{E}_1(f,g):=\mathcal{E}(f,g)+\left\langle f, g\right\rangle_{L_2(X,m)}$ the space $\mathcal{F}$ becomes a Hilbert space. The notation
\[\mathcal{B}:=\mathcal{F}\cap L_\infty(X,m)\]
will be used to denote the space of $m$-a.e. bounded energy finite functions on $X$. Endowed with the norm $\left\|f\right\|_\mathcal{B}:=\mathcal{E}_1(f)^{1/2}+\esssup_X|f|$ it becomes a Banach algebra. In particular, 
\[\mathcal{E}(f,g)\leq \left\|f\right\|_\mathcal{B}\left\|g\right\|_\mathcal{B}\ ,\ \ f,g\in\mathcal{B}.\] 
Obviously $\mathcal{B}$ is dense in $\mathcal{F}$.

The regularity of $(\mathcal{E},\mathcal{F})$ implies that for any $f,g\in\mathcal{B}$ there is a unique finite signed Radon measure $\Gamma(f,g)$ on $X$ such that
\begin{equation}\label{E:energymeasure}
2\int_X\varphi d\Gamma(f,g)=\mathcal{E}(\varphi f,g)+\mathcal{E}(\varphi g, f)-\mathcal{E}(fg,\varphi),
\end{equation}
for all $\varphi\in C(X)\cap \mathcal{F}$. $\Gamma(f,g)$ is called the \emph{mutual energy measure of $f$ and $g$}, cf. \cite{CF,FOT, LeJan78}. The mapping $(f,g)\mapsto \Gamma(f,g)$ is symmetric and bilinear on $\mathcal{B}$. Furthermore, $\Gamma(f)\geq 0$ for any $f\in\mathcal{B}$. To define the energy measure of a general element $f\in\mathcal{F}$, let $(f_n)_n\subset \mathcal{B}$ be a sequence approximating $f$ in $\mathcal{F}$ and set
\[\Gamma(f)(\varphi):=\lim_n\int_X\varphi d\Gamma(f_n),\ \ \varphi\in C(X)\cap \mathcal{F}.\]
Since
\[|\left(\int_X\varphi d\Gamma(g_1)\right)^{1/2}-\left(\int_X\varphi d\Gamma(g_2)\right)^{1/2}|\leq \left\|\varphi\right\|_{L_\infty(X,m)}^{1/2} \mathcal{E}(g_1-g_2)^{1/2}\]
for any $g_1, g_2\in\mathcal{B}$ and $\varphi\in C(X)\cap \mathcal{F}$, cf. \cite[Section 3.2]{FOT}, the functional $\Gamma(f)$ is well defined.
By regularity $\Gamma(f)$ extends to a positive linear functional on $C(X)$ and can be represented as
\[\Gamma(f)(\varphi)=\int_X\varphi d\Gamma(f),\ \ \varphi\in C(X),\]
with a uniquely determined finite and nonnegative Radon measure $\Gamma(f)$ on $X$.  By polarization we obtain mutual energy measures $\Gamma(f,g)$ for $f,g\in\mathcal{F}$ and clearly $\Gamma(f,g)(X)=\mathcal{E}(f,g)$. The Cauchy-Schwarz inequality 
\[|\Gamma(f,g)(A)|\leq \Gamma(f)(A)^{1/2}\Gamma(g)(A)^{1/2}\]
for $f,g\in\mathcal{F}$ and $A\subset X$ Borel follows from standard arguments. 

\begin{remark}\label{R:specialstandard}
Another consequence of regularity together with our topological assumptions is that $C(X)\cap\mathcal{F}$ provides a \emph{special standard core} for $\mathcal{E}$, i.e. for any compact set $K$ and any open set $U$ with $K\subset U$ there is a function $\varphi\in C(X)\cap \mathcal{F}$ such that $0\leq \varphi\leq 1$, $\varphi\equiv 1$ on $K$ and $\varphi\equiv 0$ on $U^c$. See \cite[Problem 1.4.1]{FOT}. 
\end{remark}

Let $\cpct$ denote the \emph{capacity} corresponding to $(\mathcal{E},\mathcal{F})$, given by
\[\cpct(A)=\inf\left\lbrace \mathcal{E}_1(u):u\in\mathcal{F}: u\geq 1 \text{ $m$-a.e on $A$}\right\rbrace\]
for open sets $A\subset X$ and by for general $B\subset U$,
\begin{equation}\label{E:capacity}
\cpct(B)=\inf\left\lbrace \cpct(A): \text{ $A\subset X$ open, $B\subset A$}\right\rbrace 
\end{equation}
for general sets $B\subset X$. Any set of zero capacity is a null set for $m$. A statement is said to hold q.e. (\emph{quasi everywhere}) on a subset $A\subset X$ if there exists some set $N\subset A$ with $\cpct(N)=0$ and the statement is valid for all $x\in A\setminus N$.
A Borel function $f$ is said to be \emph{quasi-continuous} if for any $\varepsilon >0$ there exists an open set $G\subset X$ such that $\cpct(G)<\varepsilon$ and $f$ is continuous on $X\setminus G$. Any function $f\in\mathcal{F}$, more precisely, any $m$-equivalence class $f$ of Borel functions in $\mathcal{F}$, possesses a Borel version (a representant of its $m$-equivalence class) $\widetilde{f}$ which is quasi-continuous. 

Let $(P_t)_{t>0}$ and $(G_\alpha)_{\alpha>0}$ be the semigroup of strongly continuous symmetric Markovian operators and the strongly continuous symmetric resolvent uniquely associated to $(\mathcal{E},\mathcal{F})$. By $Y=(Y_t)_{t\geq 0}$ we denote the $m$-symmetric Hunt process on $X$ uniquely associated with $(\mathcal{E},\mathcal{F})$, cf. \cite{FOT, MR}. As $(\mathcal{E},\mathcal{F})$ is local, $Y$ is a diffusion. For any bounded Borel function $f$ on $X$ the function $x\mapsto \mathbb{E}_x[f(Y_t)]$ provides a quasi-continuous version of $x\mapsto P_tf(x)$. We say that the semigroup $(P_t)_{t>0}$ associated with $(\mathcal{E},\mathcal{F})$ satisfies the \emph{absolute continuity condition} if $A\mapsto \mathbb{E}_x\mathbf{1}_A(x)=\mathbb{P}_x(Y_t\in A)$ is absolutely continuous with respect to $m$ for all $t>0$ and all $x\in X$. In other words, the associated transition kernels are assumed to be absolutely continuous.

\begin{remark}
If $(P_t)_{t>0}$ is a Feller semigroup (that is, as $X$ is compact, if each $P_t$ maps $C(X)$ into $C(X)$) and 
\[\left\|P_t\right\|_{L_\infty(X,m)}\leq B(t)\left\|f\right\|_{L_1(X,m)}\]
for all $t>0$ and $f\in L_1(X,m)$ with some function $B:(0,\infty)\to (0,\infty)$, then it obviously satisfies the absolute continuity condition, cf. \cite[p. 262]{CKS87}.
\end{remark}

To exclude further potential theoretic difficulties we make the following additional assumption.
\begin{assumption}\label{A:abscont}
\text{The semigroup $(P_t)_{t>0}$ satisfies the absolute continuity condition.}
\end{assumption}

\begin{definition}\label{def-1}
For an arbitrary set $B\subset X$, define
\begin{equation}\label{E:reducedF}
\mathcal{F}_{B^c}:=\left\lbrace f\in\mathcal{F}: \widetilde{f}=0 \text{ q.e. on $B$}\right\rbrace.
\end{equation}
The space $\mathcal{F}_{B^c}$ is a closed subspace of $\mathcal{F}$. Denote by $\mathcal{H}^B$ its orthogonal complement in $\mathcal{F}$ with respect to $\mathcal{E}_1$ and by $\mathcal{P}_{\mathcal{H}^B}$ the orthogonal projection onto $\mathcal{H}^B$. An element $h\in\mathcal{F}$ is called \emph{harmonic in $B^c$ in the Dirichlet form sense} if $\mathcal{E}(h,\varphi)=0$ for all $\varphi\in C_0(B^c)\cap\mathcal{F}$. Every $h\in \mathcal{H}^B$ is harmonic in $B^c$, and if $B$ is closed, every function harmonic in $B^c$ in the Dirichlet form sense is an element of $\mathcal{H}^B$, cf. \cite[Corollary 2.3.1]{FOT}. 
\end{definition} 

We recall some properties of $\mathcal{P}_{\mathcal{H}^B}$ from \cite{FOT} and a maximum principle. Under Assumption \ref{A:abscont} both follow easily.

\begin{proposition}\label{P:Dirichlet}
Let $B\subset X$ be a Borel set.
\begin{enumerate}
\item[(i)] If for some $u\in\mathcal{F}$ and some constant $c\geq 0$ we have $|\widetilde{u}|\leq c$ q.e. then $|\widetilde{(\mathcal{P}_{\mathcal{H}^B}u)}|\leq c$ q.e. Moreover, $\widetilde{(\mathcal{P}_{\mathcal{H}^B}u)}=\widetilde{u}$ q.e. on $B$.
\item[(ii)] If $h:X\to\mathbb{R}$ is a q.e. bounded Borel function with $m$-equivalence class harmonic in $B^c$ in the Dirichlet form sense, then for q.e. $x\in B^c$ we have
\[\inf_{y\in B} h(y)\leq h(x)\leq \sup_{y\in B}h(y).\]
\end{enumerate}
\end{proposition}

Proposition \ref{P:Dirichlet} follows from known potential theoretic results, and we sketch them briefly in a probabilistic way.  
However, we would like to emphasize that probabilistic methods are not substantially used in the present paper. 

The \emph{first hitting time} of a Borel set $B\subset X$ by the Hunt process $Y$ is defined as
\[\sigma_B:=\inf\left\lbrace t>0:Y_t\in B\right\rbrace.\]
A set $N\subset X$ is called \emph{polar} if there exists a Borel set $N_1\supset N$ such that $\mathbb{P}_x(\sigma_{N_1}<\infty)=0$ for all $x\in X$. By Assumption \ref{A:abscont} together with \cite[Theorems 4.1.2 and 4.2.1]{FOT} a set $N\subset X$ is polar if and only if $\cpct(N)=0$. To see Proposition \ref{P:Dirichlet} (i) consider $H_Bf(x):=\mathbb{E}^x\left[e^{-\sigma_B}f(Y_{\sigma_B})\right]$, well-defined for any Borel function $f$. For $u\in\mathcal{F}$ the function $H_B\widetilde{u}$ is a quasi continuous version of $\mathcal{P}_{\mathcal{H}^B}u$ by \cite[Theorem 4.3.1]{FOT}. Next, note that for any such $u$ we have 
\begin{equation}\label{E:cache}
H_B\widetilde{u}=\widetilde{u} \ \text{ q.e. on $B$}.
\end{equation}
For $1$-excessive functions (\ref{E:cache}) follows from \cite[Lemma 4.3.1]{FOT}, in particular (4.3.4). Recall that a Borel function $f$ is said to be \emph{$1$-excessive} if  $f(x)\geq e^{-t}P_tf(x)$ and $f(x)=\lim_{t\to 0}e^{-t}P_tf(x)$ for $m$-a.a. $x\in X$. To verify (\ref{E:cache}) for general $u\in\mathcal{F}$ we may proceed as in the proof of \cite[Theorem 4.3.1]{FOT}: For any bounded Borel function $f$ on $X$ and any $\beta>0$ consider
\[x\mapsto R_\beta f(x):=\mathbb{E}_x\int_0^\infty e^{-\beta t}f(Y_t)dt.\] 
$R_\beta f$ provides a quasi-continuous versions of $G_\beta f$. For bounded $u$ and $\beta>0$, $R_\beta \widetilde{u}$ is the difference of two bounded $1$-excessive functions and therefore $H_B(\beta R_\beta\widetilde{u})=\beta R_\beta\widetilde{u}$ q.e. on $B$. 
By \cite[Lemma 4.2.2 (ii)]{FOT} there exists a set $N_0$ of zero capacity such that for all $x\in N_0^c$ we have $\lim_{\beta\to\infty}\beta R_\beta\widetilde{u}(x)=\widetilde{u}(x)$ and $\lim_{\beta\to\infty}H_B(\beta R_\beta\widetilde{u})(x)=H_B\widetilde{u}(x)$. For nonnegative $u\in\mathcal{F}$ identity (\ref{E:cache}) follows by cutting off and using monotone convergence and for general $u\in\mathcal{F}$ by considering $u\vee 0$ and $-u\vee 0$, both elements of $\mathcal{F}$. Now Proposition \ref{P:Dirichlet} (i) follows because $H_B\widetilde{u}(x)=\mathbb{E}^x[e^{-\sigma_B}(\widetilde{u}\mathbf{1}_{N^c})(Y_{\sigma_B})]$ for any set $N$ that has zero capacity and therefore is polar. What concerns Proposition \ref{P:Dirichlet} (ii), note that \cite[(the easier part of) Proposition 2.5]{BBKT} tells that $h$ is harmonic in the probabilistic sense and therefore we have $h(x)=\mathbb{E}^x[h(Y_{t\wedge \tau_D})]$ for any $t>0$, where $D$ is an arbitrary relatively open subset of $B^c$ and 
\[\tau_D:=\inf\left\lbrace t\geq 0: Y_t\in D^c\right\rbrace\]
the \emph{first exit time} of $Y$ from $D$. If $N$ is a set of zero capacity such that $|h|<c$ on $N^c$, then by the polarity of $N$ we have $h(x)=\mathbb{E}^x[(h\mathbf{1}_{N^c})(Y_{t\wedge \tau_D})]$, and the desired result follows using bounded convergence.

\section{Locally constant functions and Dirichlet subdomains}\label{S:subdomains}

This section is concerned with Dirichlet subdomains $\mathcal{F}^F\subset \mathcal{F}$ that are constructed from functions that are locally constant on the neighborhood of a closed set $F$ of topological dimension zero. In Section \ref{S:locally} these subdomains will be used to obtain constructive descriptions for spaces of locally exact and locally harmonic $1$-forms on $X$. Since $X$ is topologically one-dimensional, every finite open cover has a finite refinement consisting of open sets with topologically zero-dimensional boundary. Recall that a compact subset of $X$ is topologically zero dimensional if and only if it is a totally disconnected set, see e.g. \cite{Eng78}. The next lemma (and therefore also Proposition \ref{P:continuous} below) uses this last fact in an essential way.

\begin{lemma} \label{L:partition}
Let $F\subset X$ be a compact set of topological dimension zero. For any $\delta>0$ there exists a finite collection $\mathcal{K}_\delta=\left\lbrace K_i\right\rbrace_{i=1}^N$ of disjoint compact sets $K_i\subset X$ of diameter less than $\delta$ such that $F=\bigcup_{i=1}^N K_i$.
\end{lemma}
\begin{proof} There exists a finite $\frac\delta2$-net $\{x_i\}_{i=1}^N$ of points in $F$. Define $K_1$ as the set points of $F$ at the distance at most $\frac\delta2$ from $x_1$. Since $F$ is   
totally disconnected,  $F\setminus K_1$ is compact. Hence we can define $K_2$ as the set points of $F\setminus K_1$ at the distance at most $\frac\delta2$ from $x_2$, and so on. 
\end{proof}

Let $F\subset X$ be closed. A function is said to be \emph{locally constant (q.e.) on an open cover of $F$} if it is constant (q.e.) on any connected component of the union of all sets in the cover. Consider the spaces
\begin{multline}
\mathcal{S}^F:=\left\lbrace f\in \mathcal{F}\cap L_\infty(X,m): \text{ there exists a finite open cover of $F$} \right.\notag\\
\left. \text{such that $\widetilde{f}$ is locally constant q.e. on this cover}\right\rbrace 
\end{multline}
and
\[\mathcal{S}^F_c:=C(X)\cap \mathcal{S}^F.\]

The next proposition is technically most involved result of this section. 

\begin{proposition}\label{P:continuous}
Let $F\subset X$ be a closed set of topological dimension zero.
 Then the space $\mathcal{S}^F$ is dense in $L_2(X,m)$. 
\end{proposition}

\begin{proof} 
We first consider $f\in C(X)\cap\mathcal{F}$. $X$ being compact, such a function $f$ is uniformly continuous on $X$. Let $\varepsilon>0$. Choose $\delta>0$ sufficiently small such that 
\begin{equation}\label{E:epsdelta}
|f(p)-f(q)|<\varepsilon/2 \ \ \text{ whenever $d(p,q)<\delta$ for two points $p,q\in X$}.
\end{equation}
Let $\mathcal{K}_\delta=\left\lbrace K_i\right\rbrace_{i=1}^N$ be a finite partition of $F$ into compact sets $K_i$ of diameter less than $\delta$ according to Lemma \ref{L:partition}. If $\varrho>0$ is the minimum distance between two of the sets $K_i$, let $0<\gamma<\delta\wedge\varrho/3$,
\begin{equation}\label{E:Vi}
V_i:=\left\lbrace x\in X: \dist(x,K_i)<\gamma/2\right\rbrace
\end{equation}
and 
\begin{equation}\label{E:Wi}
W_i:=\left\lbrace x\in X: \dist(x,K_i)<\gamma\right\rbrace,
\end{equation}
$i=1,...,N$. Clearly the $W_i$ are disjoint and $\overline{V_i}\subset W_i$. We write $V:=\bigcup_{i=1}^N\overline{V_i}$ and $W:=\bigcup_{i=1}^N W_i$.

Choose nonnegative functions $\varphi_i\in C(X)\cap \mathcal{F}$ compactly supported in $W_i$, respectively, and such that 
\[\varphi_i\equiv \frac{1}{2}(\min_{p\in \overline{V_i}}|f(p)| + \max_{p\in \overline{V_i}}|f(p)|) \ \text{ on $\overline{V_i}$},\]
$i=1,...,N$. Let $\chi\in C(X)\cap\mathcal{F}$ be such that $0\leq \chi\leq 1$, $\chi\equiv 1$ on $W^c$ and $\supp \chi\subset V^c$. Set $\varphi:=\chi f+ \sum_{i=1}^N \varphi_i\in\mathcal{F}$ and $g:=\mathcal{P}_{\mathcal{H}^{V\cup W^c}}\varphi\in\mathcal{F}$, where we use notation from Definition~\ref{def-1}. Then we have $\widetilde{g}=\varphi_i$ q.e. on each $\overline{V_i}$ by Proposition \ref{P:Dirichlet} (i) and therefore $g\in \mathcal{S}^F$. 
Similarly $\widetilde{g}=f$ q.e. on $W^c$. Again by Proposition \ref{P:Dirichlet} (i), the function $\widetilde{g}$ is q.e. bounded. Proposition \ref{P:Dirichlet} (ii) now implies
\begin{multline}
\min_{\overline{W_i}}f\leq \min\left\lbrace \varphi_i|_{\overline{V_i}},\min_{\partial W_i} f\right\rbrace=\min_{q\in\partial W_i\cup \overline{V_i}}\widetilde{g}(q)\leq\widetilde{g}(p)\notag\\
\leq \max_{q\in\partial W_i\cup \overline{V_i}}\widetilde{g}(q)=\max\left\lbrace \varphi_i|_{\overline{V_i}},\max_{\partial W_i}f\right\rbrace\leq \max_{\overline{W_i}}f
\end{multline}
for q.e. $p\in W_i\setminus\overline{V_i}$ and any $i=1,...,N$. By (\ref{E:epsdelta}) and (\ref{E:Wi}), 
\[\max_{\overline{W_i}}f-\min_{\overline{W_i}}f\leq \varepsilon\]
for any $i$ and therefore $|f(p)-\widetilde{g}(p)|\leq \varepsilon$ for q.e. $p\in W\setminus V$, hence for q.e. $p\in X$. Consequently also 
$|f(p)-g(p)|\leq\varepsilon$ for $m$-a.a. $p\in X$. As $m$ is finite this implies the result because $C(X)\cap\mathcal{F}$ is uniformly dense in $C(X)$ and the latter is dense in $L_2(X,m)$. 
\end{proof}

\begin{remark}It is not needed in our paper, but it also can be proved that if $\mathcal{F}\subset C(X)$, then $\mathcal{S}^F_c$ is dense both in $L_2(X,m)$ and in $C(X)$ because $g\in \mathcal{S}^F_c$ such that the preceding estimates hold not just quasi everywhere but everywhere. 
\end{remark}

Now let $\mathcal{F}^F$ denote the $\mathcal{E}_1$-closure of $\mathcal{S}^F$. As $\mathcal{S}^F$ is dense in $L_2(X,m)$, $(\mathcal{E},\mathcal{F}^F)$ is a local Dirichlet form on $L_2(X,m)$. If $\mathcal{F}\subset C(X)$, then $\mathcal{S}^F_c$ is also dense in $L_2(X,m)$, and denoting its $\mathcal{E}_1$-closure by $\mathcal{F}^F_{c}$ we obtain again a local the Dirichlet form $(\mathcal{E},\mathcal{F}^F_c)$, in this case even seen to be regular. In general the inclusion $\mathcal{F}^F\subset\mathcal{F}$ is proper. However, considering two disjoint closed zero-dimensional sets, the entire Dirichlet domain $\mathcal{F}$ can be reconstructed.

\begin{theorem}\label{T:reconstruct}
Let $F_1$ and $F_2$ be compact subsets of $X$ and $F_1\cap F_2=\emptyset$. Then 
\[\mathcal{F}=\overline{\mathcal{F}^{F_1}+\mathcal{F}^{F_2}}.\]
\end{theorem}

The proof of the theorem relies on the following lemmas.

\begin{lemma}\label{L:covertwo}
Let $F_1$ and $F_2$ be disjoint closed subsets of $X$. Then there exists an open cover $\left\lbrace U_1, U_2\right\rbrace$ of $X$ such that $F_1\subset U_1\setminus\overline{U_2}$ and $F_2\subset U_2\setminus \overline{U_1}$.
\end{lemma}

\begin{proof}
Let $V_1$ and $V_2$ be disjoint open sets containing $F_1$ and $F_2$, respectively. $V_1^c$ and $F_1$ are disjoint, too, hence there exist disjoint open neighborhoods $W_1$ of $V_1^c$ and $W_2$ of $F_1$. Obviously $\left\lbrace V_1,W_1\right\rbrace$ covers $X$. Since $\overline{W_1}\cap W_2=\emptyset$, we have $F_1\subset V_1\setminus \overline{W_1}$. Since $\overline{V_1}\cap V_2=\emptyset$, $F_2\subset W_1\setminus \overline{V_1}$.
\end{proof}

To any finite open cover we can associate an energy finite partition of unity:

\begin{lemma}\label{L:partition2}
For any finite open cover $U_1,...,U_N$ of $X$ there exist functions $\varphi_i\in C(X)\cap \mathcal{F}$, $i=1,...,N$, such that $0\leq\varphi_i\leq 1$, $\supp\varphi_i\subset U_i$ and $\sum_{i=1}^N\varphi_i(x)=1$.
\end{lemma}

\begin{proof}
As $X$ is a normal space, we can find open sets $V_1,...,V_N$ such that $\overline{V}_i\subset U_i$ for all $i$ and still $X=\bigcup_{i=1}^N V_i$, see e.g. \cite[Proposition B.1]{Lee}. For any $i$ let $\psi_i\in C(X)\cap \mathcal{F}$ be a function according to Remark \ref{R:specialstandard} such that $\supp\psi_i\subset U_i$, $0\leq \psi_i\leq 1$ and $\psi_i\equiv 1$ on $\overline{V_i}$. Now set $\varphi_i=\psi_i(\sum_{i=1}^N\psi_i)^{-1}$.
\end{proof}

Now Theorem \ref{T:reconstruct} follows from the preceding two lemmas together with 
\cite[Theorem 1.4.2 (iii)]{FOT}.

\section{Locally exact and locally harmonic $1$-forms}\label{S:locally}

This section first recalls the definition of $1$-forms and derivations based on Dirichlet forms as proposed by Cipriani and Sauvageot, \cite{CS03,CS09}, and then turns to related notions of local exactness and harmonicity.
 
In \cite{CS03} and \cite{CS09} the following construction of differential $1$-forms has been developed. It may be considered for any symmetric local regular Dirichlet form on a locally compact second countable Hausdorff space. Endow the vector space $\mathcal{B}\otimes\mathcal{B}_b(X)$ of simple tensors with the symmetric bilinear form
\begin{equation}\label{E:CSnorm}
\left\langle a\otimes b, c\otimes d\right\rangle_\mathcal{H}=\int_X bd\:d\Gamma(a,c),
\end{equation}
$a\otimes b, c\otimes d\in\mathcal{B}\otimes\mathcal{B}_b(X)$. Let $\left\|\cdot\right\|_\mathcal{H}$ denote the associated norm and 
\begin{equation}\label{E:kernel}
ker\:\left\|\cdot\right\|_\mathcal{H}:=\left\lbrace \sum_i a_i\otimes b_i\in\mathcal{B}\otimes\mathcal{B}_b(X): \left\|\sum_i a_i\otimes b_i\right\|_\mathcal{H}=0\right\rbrace 
\end{equation} 
(with finite sums). We write $\mathcal{H}$ for the completion of $\mathcal{B}\otimes\mathcal{B}_b(X)/ker\:\left\|\cdot\right\|_\mathcal{H}$ with respect to $\left\|\cdot\right\|_\mathcal{H}$. Obviously $\mathcal{H}$ is a Hilbert space. We refer to it as the \emph{space of $1$-forms} on $X$. The definitions 
\[c(a\otimes b):=(ac)\otimes b - c\otimes(ab)\]
and 
\[(a\otimes b)d:=a\otimes (bd)\]
for $a\otimes b\in\mathcal{B}\otimes\mathcal{B}_b(X)$, $c\in\mathcal{B}$ and $d\in\mathcal{B}_b(X)$ extend continuously to uniformly bounded actions on $\mathcal{H}$ with
\begin{equation}\label{E:boundedactions}
\left\| c(a\otimes b)\right\|_\mathcal{H}\leq \sup_X|c|\left\|a\otimes b\right\|_\mathcal{H}\ \text{ and } \ \left\| (a\otimes b)d\right\|_\mathcal{H}\leq \sup_X|d|\left\|a\otimes b\right\|_\mathcal{H} ,
\end{equation}
turning $\mathcal{H}$ into a bimodule. Using the locality of $(\mathcal{E},\mathcal{F})$ it can be shown that the left and right action coincide, see for instance \cite{H11} or \cite{IRT}.

A \emph{derivation operator} $\partial: \mathcal{B}\to \mathcal{H}$ can be defined by setting
\[\partial f:= f\otimes \mathbf{1}.\]
It satisfies the Leibniz rule,
\begin{equation}\label{E:Leibniz}
\partial(fg)=f\partial g + g\partial f, \ \ f,g \in \mathcal{B},
\end{equation}
and is a bounded linear operator satisfying 
\begin{equation}\label{E:normandenergy}
\left\|\partial f\right\|_\mathcal{H}^2=\mathcal{E}(f), \ \ f\in\mathcal{B}.
\end{equation}
On Euclidean domains and on smooth manifolds the operator $\partial$ coincides with the classical exterior derivative (in $L_2$-sense). For more detailed information we refer the reader to \cite{CS03, CS09} and to the papers \cite{CGIS11,H11,HKT2012,HKT2012,
HRT,IRT}, where this approach has been taken further in various respects.

Similarly as in \cite[Section 2]{HRT} we can extend the measure-valued bilinear mapping $\Gamma$ on $\mathcal{B}$ defined in (\ref{E:energymeasure})
to a functional-valued bilinear mapping $\Gamma_\mathcal{H}$ on $\mathcal{H}$. Setting 
\begin{equation}\label{E:GammaH}
\Gamma_\mathcal{H}(a\otimes b, c\otimes d):=bd\Gamma(a,c)
\end{equation}
for simple tensors $a\otimes b, c\otimes d\in\mathcal{B}\otimes\mathcal{B}_b(X)$, we obtain a bilinear measure-valued map $\Gamma_\mathcal{H}$. Given a general element $u\in\mathcal{H}$ we may approximate it in $\mathcal{H}$ by a sequence $(u_k)_k \subset \mathcal{B}\otimes\mathcal{B}_b(X)$ of finite linear combinations of simple tensors and set 
\begin{equation}\label{E:positive}
\Gamma_{\mathcal{H}}(u)(\varphi):=\lim_k\int_X\varphi\:d\Gamma_{\mathcal{H}}(u_k).
\end{equation}
for any $\varphi\in \mathcal{B}_b(X)$. From (\ref{E:CSnorm}) we easily obtain 
\begin{equation}\label{E:Gammabound}
|\Gamma_{\mathcal{H}}(u)(\varphi)|\leq \sup_x|\varphi(x)|\left\|u\right\|_{\mathcal{H}}^2,
\end{equation}
what defines a positive and bounded linear functional $\Gamma_\mathcal{H}(u)$ on $\mathcal{B}_b(X)$. For fixed $u\in\mathcal{H}$ we may extend $\Gamma_\mathcal{H}(u)$ to a generally unbounded bilinear functional on $L_2(X,m)$ by a simple approximation. By polarization $\Gamma_\mathcal{H}$ itself defines bilinear mapping on $\mathcal{H}$.
\begin{remark}
Let $u\in\mathcal{H}$ be fixed. As $C(X)\subset \mathcal{B}_b(X)$, the Riesz representation theorem ensures the existence of a unique nonnegative Radon measure $\Gamma_\mathcal{H}(\omega)$ on $X$ such that for any 
\[\Gamma_{\mathcal{H}}(u)(\varphi):=\int_X\varphi\:d\Gamma_{\mathcal{H}}(u), \ \ \varphi\in C(X).\]
\end{remark}
Given a $1$-form $\omega\in\mathcal{H}$, we refer to the support of the measure $\Gamma_\mathcal{H}(\omega)$ as the \emph{support of $\omega$}, cf. \cite[Section 2]{HRT}.
\begin{lemma}\label{L:support}
If $f\in\mathcal{B}$ and $U\subset X$ open are such that $\supp f\subset U$ then also $\supp\partial f\subset U$.
\end{lemma}
\begin{proof}
By definition $\Gamma_\mathcal{H}(\partial f)=\Gamma(f)$. For arbitrary $V\subset X$ open, $\mathbf{1}_V$ may be approximated pointwise $m$-a.e. by a sequence $(\varphi_n)_n \subset C(X)\cap \mathcal{F}$ of functions $\varphi_n$ which are zero on $V^c$. If $V\subset U^c\subset (\supp f)^c$ then also $\supp\varphi_n\subset \overline{V}\subset U^c$ and therefore
\[\Gamma(f)(V)\leq \liminf_n\int_X\varphi_n d\Gamma(f)=\liminf_n\left(2\mathcal{E}(\varphi_n f,\varphi_n)-\mathcal{E}(f^2,\varphi_n)\right)=0\]
by Fatou's lemma and the locality of $(\mathcal{E},\mathcal{F})$.
\end{proof}
Lemma \ref{L:support} will be used in this section. Another application of $\Gamma_\mathcal{H}$ will be seen in Section~\ref{S:NS}.

Recall that in the present paper we have assumed $X$ to be a compact, connected and topologically one-dimensional space. From now on we work under the following additional assumption.

\setcounter{section}{4}
{\begin{assumption}\label{A:assumption}
There is a topological base $\mathcal{O}$, stable under taking finite unions of sets, such that 
for any $V\in\mathcal{O}$ and any $f\in\mathcal{F}$ there is a sequence of functions $(f_n)_n\subset \mathcal{S}^{\partial V}$ such that $\lim_n\mathcal{E}(f-f_n)=0$.
\end{assumption}}
All subsequent statements this section and the following sections are stated conditionally on this assumption.

We consider the space 
\begin{equation}\label{E:Sloc}
\mathcal{S}_{loc}:=\lin\left\lbrace \partial f\mathbf{1}_V: \text{ $V\subset X$ open, $\partial V$ zero-dimensional and $f\in\mathcal{S}^{\partial V}$}\right\rbrace.
\end{equation}

\begin{theorem}\label{T:H}
The space $\mathcal{S}_{loc}$ is dense in $\mathcal{H}$.
\end{theorem}

The theorem is a consequence of the following two lemmas.

\begin{lemma}\label{L:simple}
Let $U\subset X$ be open and $f\in\mathcal{B}$. Then $f\otimes\mathbf{1}_U\in\clos(\mathcal{S}_{loc})$. 
\end{lemma}

\begin{proof}
As $X$ is second countable, we may assume the base $\mathcal{O}$ in Assumption \ref{A:assumption} is countable. Consequently $U=\bigcup_{i=1}^\infty V_i$ with certain $V_i\in\mathcal{O}$. Let $\varepsilon>0$ be given. Setting $U_N:=\bigcup_{i=1}^N V_i$ we have 
\[\left\|f\otimes \mathbf{1}_U-f\otimes\mathbf{1}_{U_N}\right\|_{\mathcal{H}}=\Gamma(f)(U\setminus U_N)^{1/2}<\frac{\varepsilon}{2},\]
provided $N$ is sufficiently large. On the other hand Assumption \ref{A:assumption} ensures that $U_N\in\mathcal{O}$ for any $N$ and there exist functions $f^{(N)}\in\mathcal{S}^{\partial U_N}$ with 
\[\left\|f\otimes\mathbf{1}_{U_N}-f^{(N)}\otimes \mathbf{1}_{U_N}\right\|_{\mathcal{H}}\leq    \mathcal{E}(f-f^{(N)})^{1/2}<\frac{\varepsilon}{2}.\]
Consequently, we have 
\[\left\|f\otimes\mathbf{1}_{U}-f^{(N)}\otimes\mathbf{1}_{U_N}\right\|_{\mathcal{H}}<\varepsilon\]
for any large enough $N$.
\end{proof}

\begin{lemma}\label{L:H11} We have
\[\mathcal{H}=\clos\lin\left\lbrace f\otimes \mathbf{1}_U: \text{ $U\subset X$ open and $f\in\mathcal{B}$}\right\rbrace.\]
\end{lemma}
This is a version of \cite[Theorem 4.1]{H11}. The proof carries over from there.
 
Now Theorem \ref{T:H} is immediate from Lemma \ref{L:simple} and Lemma \ref{L:H11}. 
 
As $\partial$ is closed and $\mathcal{E}$ admits a spectral gap, the image $Im\:\partial $ of $\mathcal{F}$ under the derivation $\partial$ is easily seen to be a closed linear subspace of the Hilbert space $\mathcal{H}$. This yields the orthogonal decomposition
\begin{equation}\label{E:Hodge}
\mathcal{H}=Im\:\partial \oplus \mathcal{H}^1(X),
\end{equation}
where we write $\mathcal{H}^1(X)$ to denote the orthogonal complement $(Im\:\partial)^\bot$ of $Im\:\partial$. For certain classes of fractal spaces (\ref{E:Hodge}) has been investigated in \cite{CGIS11,CS09} and \cite{IRT}, for harmonic spaces in \cite{H11}. To the elements of $Im\:\partial$ we refer as \emph{exact $1$-forms}. Clearly $Im\:\partial$ is nontrivial. Whether $\mathcal{H}^1(X)$ is nontrivial or not depends on the (global) topology of $X$, see \cite{IRT} and Section \ref{S:nontrivial} below. On the other hand (\ref{E:Hodge}) reminds of the classical Hodge decomposition for differential forms on smooth $1$-dimensional manifolds, which is formulated in terms of local first order operators, cf. \cite{Jost,W71}. The next definition introduces the key notions of local exactness and local harmonicity. They provide some 'localized' way of testing whether a given $1$-form belongs to $Im\:\partial$ or $\mathcal{H}^1(X)$.

\begin{definition}\label{D:locexact}
A $1$-form $\omega\in\mathcal{H}$ is called \emph{locally exact} if there exist a finite open cover $\mathcal{U}=\left\lbrace U_\alpha\right\rbrace_{\alpha\in J}$ of $X$ and functions $f_\alpha\in\mathcal{B}$, $\alpha\in J$, such that
\[\omega\mathbf{1}_{U_\alpha}=\partial f_\alpha\mathbf{1}_{U_\alpha}, \ \alpha\in J .\]
A $1$-form $\omega\in\mathcal{H}$ is called \emph{locally harmonic} if there exist a finite open cover $\mathcal{U}=\left\lbrace U_\alpha\right\rbrace_{\alpha\in J}$ of $X$ and functions $h_\alpha\in\mathcal{B}$, $\alpha\in J$, such that each $h_\alpha$ is harmonic on $U_\alpha$ (in the Dirichlet form sense) and
\[\omega\mathbf{1}_{U_\alpha}=\partial h_\alpha\mathbf{1}_{U_\alpha}, \ \alpha\in J .\]
\end{definition}
These defining properties carry over to finite sums.
\begin{lemma}
Finite linear combinations $\omega=\sum_{i=1}^N \omega_i$ of locally exact (locally harmonic) $1$-forms $\omega_i$ on $X$ are again locally exact (locally harmonic).
\end{lemma}
The simple proof by refinement is left to the reader.

Now let 
\[\mathcal{H}_{loc}:=\clos\lin\left\lbrace \omega\in\mathcal{H}: \text{ $\omega$ locally exact}\right\rbrace\]
and
\[\mathcal{H}_{loc}^1:=\clos\lin\left\lbrace \omega\in\mathcal{H}: \text{ $\omega$ locally harmonic}\right\rbrace\]
denote the spaces of limits of locally exact and locally harmonic $1$-forms on $X$, respectively. Obviously $\mathcal{H}^1_{loc}\subset \mathcal{H}_{loc}\subset \mathcal{H}$.  

\begin{lemma}\label{L:inclusionlocal}
The space $\mathcal{H}_{loc}^1$ is contained in $\mathcal{H}^1(X)$.
\end{lemma}
\begin{proof} Given $\omega\in\mathcal{H}_{loc}^1$, let $\mathcal{U}=\left\lbrace U_\alpha\right\rbrace_{\alpha\in J}$ be a finite open cover of $X$ and $h_\alpha\in\mathcal{B}$ functions harmonic in $U_\alpha$, respectively, such that $\omega\mathbf{1}_{U_\alpha}=h_\alpha\mathbf{1}_{U_\alpha}$ for each $\alpha$. Now consider an arbitrary $f\in\mathcal{B}$. Let $\left\lbrace\varphi_\alpha\right\rbrace_{\alpha\in J}\subset C(X)\cap\mathcal{F}$ be an energy finite partition of unity subordinate to $\mathcal{U}$, i.e. $\varphi_\alpha\in C_0(U_\alpha)\cap\mathcal{F}$, $0\leq \varphi_\alpha\leq 1$ and $\sum_{\alpha\in J}\varphi_\alpha=1$, cf. Lemma \ref{L:partition2}. Then $\varphi_\alpha f\in\mathcal{B}$, $\supp(\varphi_\alpha f)\subset U_\alpha$ and by Lemma \ref{L:support} also $\supp(\partial(\varphi_\alpha f))\subset U_\alpha$ for any $\alpha$. Consequently we have
\begin{align}
\left\langle \partial(\varphi_\alpha f),\omega\right\rangle_\mathcal{H}=\left\langle \partial (\varphi_\alpha f),\partial h_\alpha\mathbf{1}_{U_\alpha}\right\rangle_\mathcal{H} &=\left\langle \partial (\varphi_\alpha f),\partial h_\alpha\right\rangle_\mathcal{H}\notag\\
&=\mathcal{E}(\varphi_\alpha f, h_\alpha)\notag\\
& =0,\notag
\end{align}
and summing over $\alpha\in J$, $\left\langle \partial f,\omega\right\rangle_\mathcal{H}=0$. As $f\in \mathcal{B}$ was arbitrary, $\omega$ must be an element of $(Im\:\partial)^\bot=\mathcal{H}^1(X)$.
\end{proof}

In fact, these spaces coincide.
\begin{theorem}\label{T:Hcoincide}\mbox{}
\begin{enumerate}
\item[(i)] We have $\mathcal{H}_{loc}=\mathcal{H}$, i.e. the locally exact $1$-forms are dense in $\mathcal{H}$. 
\item[(ii)] We have $\mathcal{H}^1_{loc}=\mathcal{H}^1(X)$, i.e. the locally harmonic $1$-forms are dense in $\mathcal{H}^1(X)$.
\end{enumerate}
\end{theorem}

To prove Theorem \ref{T:Hcoincide} we consider the spaces $\mathcal{S}_{loc}$ as defined in (\ref{E:Sloc}) and
\[\mathcal{S}^1_{loc}:=\mathcal{P}_{\mathcal{H}^1}(\mathcal{S}_{loc}),\]
where $\mathcal{P}_{\mathcal{H}^1}$ denotes the orthogonal projection of $\mathcal{H}$ onto $\mathcal{H}^1(X)$. Similarly $\mathcal{P}_{Im\:\partial}$ denotes the orthogonal projection onto $Im\:\partial$.

The next lemma is a key technical result of this section. 

\begin{lemma}\label{L:proto}
Let $V\subset X$ open with $\partial V$ zero-dimensional and $f\in\mathcal{S}^{\partial V}$. Then
\begin{enumerate}
\item[(i)] The $1$-form $\partial f\mathbf{1}_V$ is locally exact.
\item[(ii)] The $1$-form $\mathcal{P}_{\mathcal{H}^1}(\partial f\mathbf{1}_V)$ is locally harmonic.
\end{enumerate}
Consequently $\mathcal{S}_{loc}\subset\mathcal{H}_{loc}$ and $\mathcal{S}^1_{loc}\subset\mathcal{H}^1_{loc}$.
\end{lemma}

\begin{proof}
Let $\left\lbrace W_k\right\rbrace_{k=1}^N$ be a finite open cover of $\partial V$ such that $f$ is constant on each $W_k$. Put
\begin{equation}\label{E:unionW}
W:=\bigcup_{k=1}^N W_k
\end{equation}
and define $U_1:=V\cup W$ as well as $U_2:=\overline{V}^c\cup W$. Obviously $\mathcal{U}=\left\lbrace U_1,U_2\right\rbrace$ is a finite open cover of $X$. By (\ref{E:unionW}) we have
\[\Gamma(f)(W)\leq\sum_{k=1}^N\Gamma(f)(W_k)=0.\]
Therefore $\left\|\partial f\mathbf{1}_{U_1}-\partial f\mathbf{1}_V\right\|_{\mathcal{H}}^2=\Gamma(f)(U_1\setminus V)\leq \Gamma(f)(W)=0$
and
\begin{equation}\label{E:inside}
(\partial f\mathbf{1}_V)\mathbf{1}_{U_1}=\partial f\mathbf{1}_V=\partial f\mathbf{1}_{U_1}.
\end{equation}
Similarly
\begin{equation}\label{E:outside}
(\partial f\mathbf{1}_V)\mathbf{1}_{U_2}=(\partial f\mathbf{1}_{U_2})\mathbf{1}_V=(\partial f\mathbf{1}_{\overline{V}^c})\mathbf{1}_V=0,
\end{equation}
and (i) follows. We turn to (ii). Let $g\in\mathcal{B}$ be a function such that 
\[\partial g=\mathcal{P}_{Im\:\partial}(\partial f\mathbf{1}_{V}).\] 
Clearly $\omega:=\mathcal{P}_{\mathcal{H}^1}(\partial f\mathbf{1}_V)$ is a member of $\mathcal{H}^1(X)$. By (\ref{E:Hodge}) it equals $\partial f\mathbf{1}_V-\partial g$. According to (\ref{E:inside}),
\[\omega\mathbf{1}_{U_1}=(\partial f\mathbf{1}_V-\partial g)\mathbf{1}_{U_1}=\partial(f-g)\mathbf{1}_{U_1}.\]
Given $\varphi\in C_0(U_1)\cap \mathcal{F}$, we observe
\[\mathcal{E}(\varphi,f-g)=\left\langle\partial\varphi,\partial (f-g)\right\rangle_\mathcal{H}=\left\langle \partial\varphi,\omega\right\rangle_\mathcal{H}=0\]
by (\ref{E:inside}) and (\ref{E:Hodge}). Hence $f-g$ is harmonic on $U_1$. Similarly,
\[\omega\mathbf{1}_{U_2}=(\partial f \mathbf{1}_V-\partial g)\mathbf{1}_{U_2}=-\partial g\mathbf{1}_{U_2}\]
by (\ref{E:outside}), and for $\varphi\in C_0(U_2)\cap \mathcal{F}$ we have
\[\mathcal{E}(\varphi,g)=-\left\langle \partial\varphi, \omega\right\rangle_\mathcal{H}=0,\]
so $g$ is harmonic on $U_2$. Consequently $\omega$ is locally harmonic.
\end{proof}

Recall that according to Theorem \ref{T:H} the space $\mathcal{S}_{loc}$ is dense in $\mathcal{H}$.

\begin{lemma}\label{L:protodense}\mbox{}
The space $\mathcal{S}_{loc}^1$ is dense in $\mathcal{H}^1(X)$.
\end{lemma}

\begin{proof}
Let $\eta\in\mathcal{H}^1\subset \mathcal{H}$. According to (i) there is a sequence $(\eta^{(n)})_n\subset \mathcal{S}_{loc}$ approximating $\eta$ in $\mathcal{H}$. By (\ref{E:Hodge}) also $\left( \mathcal{P}_{\mathcal{H}^1}(\eta^{(n)})\right)_n$ approximates $\eta$ in $\mathcal{H}$, and $\left( \mathcal{P}_{\mathcal{H}^1}(\eta^{(n)})\right)_n \subset\mathcal{S}^1_{loc}$.
\end{proof}

Now Theorem \ref{T:Hcoincide} follows from Lemma \ref{L:proto} and Lemma \ref{L:protodense} because we have $\mathcal{H}\subset\mathcal{H}_{loc}$ and $\mathcal{H}^1\subset\mathcal{H}^1_{loc}$.

\section{Nontriviality of the first \v{C}ech cohomology}\label{S:nontrivial}

The main results of this section are Theorem \ref{T:nontrivial1} and  \ref{T:nontrivial2} below, which state that, roughly speaking, the (real) first \v{C}ech cohomology of $X$ is nontrivial if and only if $\mathcal{H}^1(X)$ is nontrivial.
 
For convenience and to fix notation we briefly recall some basics about \v{C}ech cohomology, \cite{BT82,ES,W71}. Let $\mathcal{U}=\left\lbrace U_\alpha\right\rbrace_{\alpha\in J}$ be a finite open cover of $X$. To $\mathcal{U}$ we assign its \emph{\v{C}ech complex} which by definition is the abstract simplicial complex consisting of all $q$-simplices spanned by $q+1$ distinct elements $\alpha_0,...,\alpha_q$ of $J$ for which
\[U_{\alpha_0\cdots\alpha_q}:=U_{\alpha_0}\cap ... \cap U_{\alpha_q}\]
is non-empty. We fix an orientation and write an ordered $q+1$ tuple $\alpha_0\cdots\alpha_q$ to denote an oriented $q$-simplex $\check{C}^q(\mathcal{U})$ denotes the vector spaces of oriented \v{C}ech $q$-cochains, i.e. functions that assign a real number to each oriented $q$-simplex. A change of orientation of a simplex changes the sign of the function value on this simplex according to the sign of the corresponding permutation of the tuple. For our purposes $q=1$ is most important, therefore we restrict attention to some special cases of more general facts.

The difference operator $d:\check{C}^0(\mathcal{U})\to \check{C}^1(\mathcal{U})$ is defined by
\[df(\alpha_0\alpha_1):=f(\alpha_1)-f(\alpha_0), \ \ f\in C^0(\mathcal{U}).\]
Its image $\check{B}^1(\mathcal{U}):=Im\: d$ is a subspace of $C^1(\mathcal{U})$, the \emph{space of $1$-coboundaries}. 
Similarly, $d:\check{C}^1(\mathcal{U})\to \check{C}^2(\mathcal{U})$ is defined by
\[df(\alpha_0\alpha_1\alpha_2):=f(\alpha_1\alpha_2)-f(\alpha_0\alpha_2)+f(\alpha_0\alpha_1), \ \ f\in C^1(\mathcal{U}).\]
Its kernel $\check{Z}^1(\mathcal{U}):=ker\:d$ is another subspace of $\check{C}^1(\mathcal{U})$, the \emph{space of $1$-cocycles}. The quotient
\[\check{H}^1(\mathcal{U}):=\check{Z}^1(\mathcal{U})/\check{B}^1(\mathcal{U})\]
is called the \emph{space of harmonic $1$-cochains} or \emph{first \v{C}ech cohomology of $\mathcal{U}$}. Given a cocycle $c\in \check{Z}^1(\mathcal{U})$, we denote its cohomology class in $\check{H}^1(\mathcal{U})$ by $[c]$. An open cover $\mathcal{V}=\left\lbrace V_\beta\right\rbrace_{\beta\in I}$ is a refinement of $\mathcal{U}=\left\lbrace U_\alpha\right\rbrace_{\alpha\in J}$, written $\mathcal{V}<\mathcal{U}$, if for each $\beta \in I$ there is some $\alpha=:\pi(\beta)\in J$ such that $V_\beta\subset U_\alpha$. This determines a refining map $\pi: I\to J$, which yields linear maps $\pi:\check{C}^q(\mathcal{U})\to\check{C}^q(\mathcal{V})$ by 
\[\pi(c)(\beta_0\cdots \beta_q):=f(\pi(\beta_0)\cdots\pi(\beta_q)).\]
Since $d\circ \pi=\pi\circ d$, the maps $\pi$ themselves induce well defined linear maps, often called \emph{refining maps},
\[\pi_\mathcal{U}^\mathcal{V}:\check{H}^1(\mathcal{U})\to\check{H}^1(\mathcal{V}),\ \ \pi_\mathcal{U}^\mathcal{V}([c]):=[\pi(c)].\]
We will make use of the following known fact and include a short textbook proof for convenience.
\begin{lemma}\label{L:directlimit}
The maps $\pi_\mathcal{U}^\mathcal{V}: \check{H}^1(\mathcal{U})\to\check{H}^1(\mathcal{V})$ are injective.
\end{lemma}

\begin{proof}
Assume $c\in \mathcal{Z}^1(\mathcal{U})$ is such that $\pi(c)=b(\beta')-b(\beta)\in\check{B}^1(\mathcal{V})$ with $b\in\check{C}^0(\mathcal{V})$. Being an element of $\mathcal{Z}^1(\mathcal{U})=ker\:d$, $c$ satisfies
\[c(\pi(\beta)\pi(\beta'))=c(\alpha\pi(\beta'))-c(\alpha\pi(\beta)).\]
Therefore, if we set $h(\alpha,\beta):=b(\beta)-c(\alpha\pi(\beta))$ and define 
\[h(\alpha):=h(\alpha,\beta')  \text{ whenever $U_\alpha\cap V_\beta\neq \emptyset$},\]
$h$ is a well defined element of $\check{C}^0(\mathcal{U})$. Obviously $dh=c$, hence $c\in\check{B}^1(\mathcal{U})$.
\end{proof}

Further, it can be shown that if both $\pi$ and $\sigma$ are refining maps, then $\sigma_\mathcal{U}^\mathcal{V}=\pi_\mathcal{U}^\mathcal{V}$. See \cite{ES} or \cite{W71}. Therefore the spaces $\check{H}^1(\mathcal{U})$ together with the maps $\pi_\mathcal{U}^\mathcal{V}$ form a
direct system along the set of open covers with the refinement relation $<$.
The \emph{first \v{C}ech cohomology $\check{H}^1(X)$ of the space $X$} is the corresponding direct limit 
\[\check{H}^1(X):=\varinjlim_{\mathcal{U}}\check{H}^1(\mathcal{U}).\]
Recall that the direct limit on the right hand side can be obtained by considering the disjoint union $\sqcup_{\mathcal{U}}\check{H}^1(\mathcal{U})$ taken over all possible finite open covers $\mathcal{U}$ of $X$. Two of its elements $h_1\in \check{H}^1(\mathcal{U})$ and $h_2\in \check{H}^1(\mathcal{V})$ are equivalent, $h_1\sim h_2$, if there is a finite open cover $\mathcal{W}$ such that  $\mathcal{W}<\mathcal{V}$, $\mathcal{W}<\mathcal{U}$ and $\pi_\mathcal{U}^\mathcal{W}h_1=\pi_\mathcal{V}^\mathcal{W}h_2 \in \check{H}^1(\mathcal{W})$. The direct limit $\varinjlim_{\mathcal{U}}\check{H}^1(\mathcal{U})$ then is defined as the resulting factor space $(\sqcup_{\mathcal{U}}\check{H}^1(\mathcal{U}))/\sim$.

A set $\mathbb{V}$ of open covers of $X$ is called \emph{cofinal} if every open cover $\mathcal{U}$ of $X$ has a refinement $\mathcal{V}\in\mathbb{V}$. We record another simple fact.

\begin{corollary}\label{C:nontrivial}
Let $\mathbb{V}$ be a cofinal set of open covers of $X$. If $\check{H}^1(X)$ is nontrivial then for some $\mathcal{V}\in\mathbb{V}$ the space $\check{H}^1(\mathcal{V})$ must be nontrivial.
\end{corollary}

\begin{proof}
Assume that $\check{H}^1(\mathcal{V})=\left\lbrace 0\right\rbrace$ for all $\mathcal{V}\in\mathbb{V}$. For any  $h\in\sqcup_{\mathcal{U}}\check{H}^1(\mathcal{U})$ there is some open cover $\mathcal{W}$ of $X$ such that $h\in\check{H}^1(\mathcal{W})$. As $\mathbb{V}$ is cofinal, $\mathcal{W}$ has a refinement $\mathcal{V}\in\mathbb{V}$.
By the above assumption, $\pi_\mathcal{W}^\mathcal{V}h=0$ in $\check{H}^1(\mathcal{V})$. If $\mathcal{U}$ is any further refinement of $\mathcal{V}$ then 
\[\pi_\mathcal{V}^\mathcal{U}h=\pi_\mathcal{V}^\mathcal{U}\pi_\mathcal{W}^\mathcal{V}h=0,\]
which implies that the equivalence class of $h$ in $\check{H}^1(X)$ is zero. As $h$ was arbitrary, this argument would imply $\check{H}^1(X)=\left\lbrace 0\right\rbrace$, contradicting the assumption of the lemma.
\end{proof}

Now let $\mathbb{V}_0$ be the set of all finite open covers  $\mathcal{V}=\left\lbrace V_\beta\right\rbrace_{\beta \in I}$ such that
\begin{multline}\label{E:Lebesgue}
\text{ all open sets $V_\beta$ are connected and such that}\\ 
\text{ for distinct $\beta,\beta',\beta''\in I$ we have $V_\beta\cap V_{\beta'}\cap V_{\beta''}=\emptyset.$}
\end{multline}
The one-dimensionality of $X$ implies the following.
\begin{lemma}
The set $\mathbb{V}_0$ of such covers is cofinal.
\end{lemma}
\begin{proof} 
As $X$ has Lebesgue covering dimension one, an arbitrary open cover of $X$ has a refinement satisfying the intersection property in (\ref{E:Lebesgue}). Considering the connected components of the sets contained in this refinement, we obtain a cover by connected open sets still satisfying the intersection property. Since $X$ is compact, finitely many of these sets suffice to cover $X$. Consequently any open cover of $X$ has a refinement in $\mathbb{V}_0$.
\end{proof}

To prove the theorems below we use the following characterization of exact $1$-forms.
\begin{proposition}\label{P:gradient}
Let $U\subset X$ be open and $f\in\mathcal{B}$. Then $f\otimes \mathbf{1}_U$ is exact if and only if there exists a function $g\in\mathcal{B}$ such that $\widetilde{f}=\widetilde{g}$ q.e. on $U$ and $\widetilde{g}$ is constant q.e. on each connected component of $X\setminus U$.
\end{proposition}

\begin{proof}
The if-part is clear. For the converse implication, assume there is some $g\in\mathcal{B}$ that equals $f|_U$ on $U$ q.e. and is constant q.e. on each connected component of $X\setminus U$. Then
\begin{align}
\left\|h\otimes \mathbf{1}_U-g\otimes\mathbf{1}\right\|_\mathcal{H}^2&=\left\|(h-g)\otimes\mathbf{1}_U+g\otimes\mathbf{1}_{X\setminus U}\right\|_\mathcal{H}^2\notag\\
&=\Gamma(h-g)(U)+\Gamma(g)(X\setminus U)\notag\\
&=0.
\end{align}
\end{proof}

The first main result of this section is the following.
\begin{theorem}\label{T:nontrivial1}
Let $X$ be compact and topologically one-dimensional. If $\mathcal{H}^1(X)$ is nontrivial then $\check{H}^1(X)$ is nontrivial.
\end{theorem}

\begin{proof}
Assume $\check{H}^1(X)$ is trivial. Let $\partial f\mathbf{1}_V \in \mathcal{S}_ {loc}$ be arbitrary. Recall that according to the definition of $\mathcal{S}_{loc}$, the set $V\subset X$ is open and has a zero-dimensional boundary $\partial V$ admitting a finite open cover on which $f$ is locally constant. $\partial V$ being zero-dimensional, the cover admits a finite refinement consisting of disjoint open sets $U_1,...,U_N$. 

Put $V':=V\cup \bigcup_{i=1}^N U_i$. We claim that $(V')^c$ and $\overline{V}$ are disjoint. In fact, we have $\partial V\subset \bigcup_{i=1}^N U_i$ and the sets $\partial V$ and $\partial \left(\bigcup_{i=1}^N U_i\right) \subset \bigcup_{i=1}^N \partial U_i$ are disjoint: If there were some $x\in \partial V\cap \partial U_i$ then there had to be some $k\neq i$ such that $x\in U_k$ resulting
in $U_i\cap U_k\neq \emptyset$, which is impossible. Accordingly $\partial V'$ and $\partial V$ are disjoint, what implies our previous claim.

Let $W$ be an open neighborhood of $(V')^c$ such that $W\cap V=\emptyset$. Then $\mathcal{U}:=\left\lbrace V,W,U_1,...,U_N\right\rbrace$ is an open cover of $X$. As a consequence of Lemma \ref{L:directlimit}, $\check{H}^1(\mathcal{U})$ must be trivial. Since by construction no more than two sets of $\mathcal{U}$ have non-empty intersection, this means its \v{C}ech complex is a graph and does not contain a cycle, i.e. it is a tree. Therefore the sets $U_1,..., U_N$ belong to different connected components of $X\setminus V$. But then $\widetilde{f}$ possesses a continuation to all of $X$ which is constant q.e. on each of these components and according to Proposition \ref{P:gradient} $\partial f\mathbf{1}_V=f\otimes \mathbf{1}_V$ is a gradient. 

Since $\partial f\mathbf{1}_V$ was an arbitrary element of $\mathcal{S}_{loc}$, any member of the latter space is a gradient, and its closure is $Im\:\partial=\mathcal{H}$, leaving $\mathcal{H}^1=\left\lbrace 0\right\rbrace$.
\end{proof}

Our proof of the converse implication needs an additional assumption. Let $U\subset X$ be a connected open set. A set $D\subset U$ will be called \emph{disconnecting for $U$} if $U$ decomposes into a disjoint union
\[U=U_1\cup U_2\cup D\]
with $U_1$ and $U_2$ open. In other words, the removal of $D$ turns the connected set $U$ into the disconnected set $U_1\cup U_2$. If $\mathbb{V}$ is a set of covers consisting of connected open sets, a set $D$ will be called \emph{disconnecting for $\mathbb{V}$} if it is disconnecting for any connected open set in any of the covers of $\mathbb{V}$.

\begin{theorem}\label{T:nontrivial2}
Let $X$ be compact and topologically one-dimensional. Assume that there exists a cofinal set $\mathbb{V}$ of finite open covers $\mathcal{V}=\left\lbrace V_\beta\right\rbrace_{\beta \in I}$ satisfying (\ref{E:Lebesgue}) and such that any set $D$ which is disconnecting for $\mathbb{V}$ has positive capacity, $\cpct(D)>0$. Then if $\check{H}^1(X)$ is nontrivial also $\mathcal{H}^1(X)$ must be nontrivial.
\end{theorem}

\begin{remark}\mbox{}
\begin{enumerate}
\item[(i)] If $(\mathcal{E},\mathcal{F})$ is a resistance form 
in the sense of Kigami \cite{Ki01,Ki03,Ki12} then points have positive capacity and so the hypothesis in Theorem \ref{T:nontrivial2} obviously holds. 
This is the case if, for instance,  
the spectral dimension $d_S$ exists and is less than $2$.
\item[(ii)]  Theorem \ref{T:nontrivial2} also applies to certain classes of self-similar and cell-structured sets which carry a diffusion that admits transition densities. Examples include generalized Sierpinski carpets \cite{BB99,BBKT}. Note that in these cases the spectral dimension $d_S$ may even be greater than or equal to $2$.
\end{enumerate}
\end{remark}

\begin{proof}
By Corollary \ref{C:nontrivial} there is a cover $\mathcal{V}\in\mathbb{V}$ such that $\check{H}^1(\mathcal{V})$ is nontrivial. The \v{C}ech graph of this cover $\mathcal{V}$ must contain a cycle, i.e. there must be covering sets $V_{\beta_0}, V_{\beta_1},...V_{\beta_N} \in\mathcal{V}$ such that $V_{\beta_i}\cap V_{\beta_{i+1}}\neq \emptyset$, $i=0,...,N$, where $\beta_{N+1}:=\beta_0$.
We may assume 
\begin{equation}\label{E:firstintersection}
V_{\beta_0}\cap \overline{V_{\beta_0\beta_1}}\cap \overline{V_{\beta_0\beta_N}}=\emptyset.
\end{equation} 
For if $x_1\in V_{\beta_0}\cap\overline{V_{\beta_1}}\cap \overline{V_{\beta_N}}$, then there is a base set $V_1$ containing $x$ and such that $\overline{V_1}\subset V_{\beta_0}$. If now $V_{\beta_0}\cap\overline{V_{\beta_1}\setminus V_1}\cap \overline{V_{\beta_N}}=\emptyset$ put $V_2:=\emptyset$. If there is some $x_2\in V_{\beta_0}\cap\overline{V_{\beta_1}\setminus V_1}\cap \overline{V_{\beta_N}}$, choose another base set $V_2$ containing $x_2$ and such that $\overline{V_2}\subset V_{\beta_0}$. Now proceed further by induction: If $V_{\beta_0}\cap\overline{V_{\beta_1}\setminus \bigcup_{j=1}^k V_k}\cap \overline{V_{\beta_N}}=\emptyset$ then put $V_{k+1}=\emptyset$, if a point $x_{k+1}$ is contained, choose a base set $V_{k+1}$ containing $x_{k+1}$ and such that $\overline{V_{k+1}}$ is in $V_{\beta_0}$. Finally replace $V_{\beta_1}$ by $V_{\beta_1}\setminus \bigcup_{j=1}^\infty V_j$ and denote it again by $V_{\beta_1}$. As any point of the space must be contained in one of the base sets, (\ref{E:firstintersection}) now holds. On the other hand the new system of open sets still covers $X$. Note also that this modification does not alter the set $(\bigcup_{k=1}^N V_{\beta_k})\setminus V_{\beta_0}$.

We have 
\begin{equation}\label{E:secondintersection}
\partial V_{\beta_0}\cap \overline{V_{\beta_0\beta_1}}\cap \overline{V_{\beta_0\beta_N}}=\emptyset:
\end{equation}
Clearly the set $\partial V_{\beta_0}\cap V_{\beta_0\beta_1}\cap \overline{V_{\beta_0\beta_N}}$ is empty. For any $x\in \partial V_{\beta_0}\cap \partial V_{\beta_0\beta_1}\cap V_{\beta_0\beta_N}$ we could find an open neighborhood $V_x\subset V_{\beta_0\beta_2}$ which intersects $V_{\beta_0\beta_2}$, which contradicts (\ref{E:Lebesgue}). Any $x\in \partial V_{\beta_0}\cap \partial V_{\beta_0\beta_1}\cap \partial V_{\beta_0\beta_N}$ would have to be contained in some other set $V_{\beta'}$ of the cover, causing a similar contradiction.

Clipping (\ref{E:firstintersection}) and (\ref{E:secondintersection}), $\overline{V_{\beta_0\beta_1}}$ and $\overline{V_{\beta_0\beta_N}}$ are seen to be disjoint compact subsets of $X$. Therefore we can find an open set $W\supset \overline{V_{\beta_0\beta_1}}$, disjoint from $\overline{V_{\beta_0\beta_N}}$, and a function $\varphi\in C_0(X)\cap\mathcal{F}$ such that $0\leq \varphi\leq 1$, $\varphi\equiv 1$ on $\overline{V_{\beta_0\beta_1}}$ and zero outside $W$. In particular, 
$\varphi\equiv 1$ on $\partial V_{\beta_0}\cap V_{\beta_1}$ and $\varphi\equiv 0$ on $\partial V_{\beta_0}\cap V_{\beta_N}$. 
These two sets belong to the same connected component $K$ of $X\setminus V_{\beta_0}$. 

Now assume there is a function $f\in\mathcal{B}$ such that for its quasi-continuous Borel version $\widetilde{f}$ we have
$\widetilde{f}=\varphi$ q.e. on $V_{\beta_0}$ and $\widetilde{f}$ is constant q.e. on any connected component of $X\setminus U$. Then there exist a set $N\subset\bigcup_{i=1}^N V_{\beta_k}$ of zero capacity and a constant $c\in\mathbb{R}$ such that $\widetilde{f}\equiv c$ on $\bigcup_{i=1}^N V_{\beta_k}\setminus (V_{\beta_0}\cup N)$, $\widetilde{f}\equiv 1$ on $V_{\beta_0\beta_1}\setminus N$ and $\widetilde{f}\equiv 0$ on $V_{\beta_0\beta_N}$. In particular, $\widetilde{f}$ must be discontinuous on at least one of the sets $(\partial V_{\beta_0}\cap V_{\beta_1})\setminus N$ and $(\partial V_{\beta_0}\cap V_{\beta_N})\setminus N$. As the sets $\partial V_{\beta_0}\cap V_{\beta_1}$ and $\partial V_{\beta_0}\cap V_{\beta_N}$ are disconnecting (for $V_{\beta_1}$ and $V_{\beta_N}$, respectively), they both are of positive capacity, what contradicts the quasi-continuity of $\widetilde{f}$.

Consequently $\partial\varphi\mathbf{1}_{V_{\beta_0}}$ is not a gradient but a nontrivial element of $\mathcal{H}^1$. 
\end{proof}

A coarse sufficient condition for the validity of capacity condition in Theorem \ref{T:nontrivial2} can be formulated in terms of the irreducibility of restricted Dirichlet forms. Given an non-empty open subset $U\subset X$, let $\mathcal{F}^U$ denote the $\mathcal{E}_1$-closure of $C_0(U)\cap \mathcal{F}$ and write 
\[\mathcal{E}^U:=\mathcal{E}|_{\mathcal{F}^U}.\]
Then $(\mathcal{E}^U,\mathcal{F}^U)$ is a regular Dirichlet form, referred to as the \emph{restriction of $(\mathcal{E},\mathcal{F})$ to $U$} (with Dirichlet boundary conditions). The associated strongly continuous symmetric Markovian semigroup on $L_2(U,m)$ is denoted by $(P_t^U)_{t\geq 0}$, often called the \emph{killed semigroup on $U$}. The symmetric Hunt process associated with $(\mathcal{E}^U,\mathcal{F}^U)$ and $(P_t^U)_{t\geq 0}$ is $(Y_t^U)_{t\geq 0}$, defined by
\[Y_t^U:=\begin{cases} Y_t,\ t<\tau_U\\ \Delta,\ t\geq \tau_U,\end{cases}\]
where $\Delta$ is the point at infinity in the one-point compactification of $X$ or, if $X$ is already compact, an adjoined isolated point. A function $f$ on $X$ or a subset of $X$ is set to be zero at $\Delta$, $f(\Delta):=0$, cf. \cite{Ch82, FOT}.
A Borel set $A\subset U$ is an \emph{invariant set for $(\mathcal{E}^U,\mathcal{F}^U)$} if 
\[P_t^U(\mathbf{1}_Af)(x)=0\ \ \text{for $m$-a.e. $x\in U\setminus A$}\]
for any $f\in L_2(U,m)$ and $t>0$. $(\mathcal{E}^U,\mathcal{F}^U)$ is called \emph{irreducible} if every invariant subset $A$ for $(\mathcal{E}^U,\mathcal{F}^U)$ is trivial, i.e. if $m(A)=0$ or $m(U\setminus A)=0$. For $(\mathcal{E},\mathcal{F})$ itself invariance and irreducibility are defined in an analogous manner. See \cite{CF,FOT} for further details.

For $U\subset X$ open let $\cpct^U$ denote the capacity with respect to $(\mathcal{E}^U,\mathcal{F}^U)$ in the sense of (\ref{E:capacity}), that is
\[\cpct^U(A)=\inf\left\lbrace \mathcal{E}_1(f):f\in\mathcal{F}^U: \widetilde{f}\geq 1 \text{ $m$-a.e on $A$}\right\rbrace\]
for $A\subset U$ open, and for general $B\subset U$,
\[\cpct^U(B)=\inf\left\lbrace \cpct^U(A): \text{ $A\subset U$ open, $B\subset A$}\right\rbrace .\]
For a Borel set $B\subset U$ let $\sigma^U_B$ denote the first hitting time of $B$ by the restricted process $Y^U$,
\[\sigma_B^U=\inf\left\lbrace t>0: Y_t^U\in B\right\rbrace.\]

\begin{lemma}\label{L:cap}
Let $B\subset U$ be a Borel set. Then $\cpct^U(B)>0$ implies $\cpct(B)>0$.
\end{lemma}

\begin{proof}
As $(\mathcal{E},\mathcal{F})$ satisfies the absolute continuity hypothesis by Assumption \ref{A:abscont}, so does $(\mathcal{E}^U,\mathcal{F}^U)$. By \cite[Theorems 4.1.2 and 4.2.1]{FOT} $B$ cannot be polar for $Y^U$, that is $\mathbb{P}_x(\sigma_B^U<\infty)>0$ for some $x\in U$. Then clearly also $\mathbb{P}_x(\sigma_B<\infty)>0$, which means that $B$ is not polar for $Y$, and the result follows by another application of the cited theorems.
\end{proof}

We obtain the following consequences of irreducibility.

\begin{proposition}\label{P:sufficient}
If $U\subset X$ is a connected open set and $(\mathcal{E}^U,\mathcal{F}^U)$ is irreducible, then any set $D$ disconnecting $U$ has positive capacity.
\end{proposition}

\begin{proof}
Assume that $U=U_1\cup U_2\cup D$ with $U_1, U_2$ open and the union being disjoint. We may assume $m(D)=0$ because $m$ charges no set of zero capacity. Since $m(U_1)>0$ and $m(U_2)>0$ neither $U_1$ nor $U_2$ can be invariant for $(\mathcal{E}^U,\mathcal{F}^U)$. Hence without loss of generality there exist some $t>0$, $f\in L_2(U,m)$ with $f\geq 0$ and some set $A\subset U_1$ with $m(A)>0$ such that 
\[\mathbb{E}_x[(\mathbf{1}_{U_2}f)(Y_t^U)]=P_t^U(\mathbf{1}_{U_2}f)(x)>0\ \text{ for all $x\in A$.}\]
In particular, for all $x\in A$ the $\mathbb{P}_x$-probability to have a connected path $Y^U([0,t])$ joining $x$ and $U_2$ is positive. Each such path necessarily hits the relative boundary of $U_2$ with respect to $U$, and this relative boundary is contained in $D$. If $\cpct^U(D)=0$ then $D$ must be polar for $Y^U$ and in particular 
\[\mathbb{P}_x(Y_t\in U_2)\leq \mathbb{P}_x(\sigma_D^U<\infty)=0\ \text{ for all $x\in A$},\]
a contradiction. Consequently $\cpct^U(D)>0$, and by Lemma \ref{L:cap} the conclusion follows.
\end{proof}

\begin{corollary}
Assume that there exists a cofinal set $\mathbb{V}$ of finite open covers $\mathcal{V}=\left\lbrace V_\beta\right\rbrace_{\beta \in I}$ satisfying (\ref{E:Lebesgue}) and such that for any open set $V_\beta$ from any cover $\mathcal{V}$ of $\mathbb{V}$, the Dirichlet form $(\mathcal{E}^{V_\beta},\mathcal{F}^{V_\beta})$ is irreducible. Then the conclusion of Theorem \ref{T:nontrivial2} holds.
\end{corollary}

\section{Form Laplacian and harmonic $1$-forms}\label{S:harmonicforms}

In this section we first recall some notions of vector calculus as proposed in \cite{HRT}. Then we give definitions for the Hodge Laplacian on $1$-forms and for harmonic $1$-forms. Finally, we define a specific functional needed in the next section to formulate Navier-Stokes type equations. 

Due to the self-duality of $\mathcal{H}$ we regard the elements of $\mathcal{H}$ also as vector fields and $\partial$ as a generalization of the classical \emph{gradient operator}. It may be viewed as an unbounded closed linear operator from $L_2(X,m)$ to $\mathcal{H}$ with domain $\mathcal{F}$. Let $\mathcal{B}^\ast$ denote the dual of $\mathcal{B}$ with the usual norm
\[\left\|u\right\|_{\mathcal{B}^\ast}=\sup\left\lbrace |u(f)|: f\in\mathcal{B}, \left\|f\right\|_\mathcal{B}\leq 1\right\rbrace.\]  
Given a vector field of form $g\partial f$, its \emph{divergence} can be defined similarly as in \cite{HRT} by
\begin{equation}\label{E:divergence}
\partial^\ast(g\partial f):=-\int_Xg\:d\Gamma(\cdot,f)\in\mathcal{B}^\ast.
\end{equation}
The map $g\partial f\mapsto \partial^\ast(g\partial f)$ extends continuously to a bounded linear operator $\partial^\ast$ from $\mathcal{H}$ into $\mathcal{B}^\ast$, as shown in \cite[Lemma 3.1]{HRT}. Note that this is a definition in a distributional sense. Seen as an unbounded operator from $\partial^\ast:\mathcal{H}\to L_2(X,m)$ with domain 
\begin{multline}
\dom\:\partial^\ast:=\left\lbrace v\in\mathcal{H}: \text{ there exists $v^\ast\in L_2(X,m)$ such that }\right.\notag\\
\left. \left\langle u,v^\ast\right\rangle_{L_2(X,m)}=-\left\langle \partial u,v\right\rangle_\mathcal{H} \text{ for all $u\in\mathcal{F}$}\right\rbrace,
\end{multline} 
the operator $-\partial^\ast$ is seen to be the adjoint of $\partial$, i.e.
\begin{equation}\label{E:IbP}
\left\langle u,\partial^\ast v\right\rangle_{L_2(X,m)}=-\left\langle \partial u, v\right\rangle_\mathcal{H},\ \ u\in\mathcal{F}.
\end{equation}
Sometimes $\partial$ is referred to as the \emph{codifferential} associated with $\partial$. Its domain $\dom\:\partial^\ast$ is dense in $\mathcal{H}$ and $(\partial^\ast, \dom\partial^\ast)$ is a closed linear operator, \cite[Theorem VIII.1]{RS}. Let $A$ denote the infinitesimal generator of $(\mathcal{E},\mathcal{F})$. For any $f\in\mathcal{B}$ we have 
\begin{equation}\label{E:generatorA}
Af=\partial^\ast\partial f
\end{equation}
in $\mathcal{B}^\ast$, see \cite[Lemma 3.2]{HRT}, and for functions $f$ from the domain $\dom\:A$ of $A$ the identity (\ref{E:generatorA}) holds in $L_2(X,m)$. It is useful to record a lemma on suitable cores.


\begin{lemma} \label{L:firstlemma}\mbox{}
There is an $\mathcal{E}_1$-dense subspace $\mathcal{C}$ of $\mathcal{F}$ such that for all $g\in\mathcal{C}$ we have $Ag\in\mathcal{B}$.
Its image $\partial(\mathcal{C})$ under the derivation $\partial$ is contained in $\dom\:\partial^\ast$.
\end{lemma}
\begin{proof}
Set $\mathcal{C}:=\left\lbrace G_1f: f\in\mathcal{F}\cap C(X)\right\rbrace$. Each $g\in \mathcal{F}$ is of the form $g=(A-I)^{-1/2}h$ with $h\in L_2(X,m)$. As the range of $G_1$ is dense in $L_2(X,m)$, the function $h$ can be approximated in the $L_2(X,m)$-norm by a sequence of functions $G_1h_n$ with $h_n\in \mathcal{F}\cap C_0(X)$ and therefore $g$ can be approximated in $\mathcal{E}_1$ by the functions $G_1h_n$. For $g=G_1f\in\mathcal{C}$ with $f\in\mathcal{F}\cap C(X)$ we have $Ag=f+G_1f\in\mathcal{F}$, and the Markov property implies
\[\left\|G_1f\right\|_{L_\infty(X,m)}\leq\int_0^\infty e^{-t} \left\|P_tf\right\|_{L_\infty(X,m)}dt\leq \left\|f\right\|_{L_\infty(X,m)}.\]
The last statement follows from (\ref{E:generatorA}).
\end{proof}

After these preliminaries we can define the Hodge Laplacian. Set
\[\dom\:\Delta_1:=\left\lbrace \omega\in \dom\:\partial^\ast: \partial^\ast\omega\in\mathcal{F}\right\rbrace.\]
\begin{definition}\label{D:formlaplace}
The operator $\Delta_1$ with domain $\dom\:\Delta_1$ on $\mathcal{H}$, given by 
\begin{equation}\label{E:formlaplace}
\Delta_1\omega:=\partial\partial^\ast\omega, \ \ \omega\in \dom\:\Delta_1.
\end{equation}
will be called the \emph{Hodge Laplacian} on $1$-forms associated with $(\mathcal{E},\mathcal{F})$.
\end{definition}

\begin{theorem}\label{T:hodgesa}
The Hodge Laplacian $(\Delta_1, \dom\:\Delta_1)$ is a self-adjoint operator on $\mathcal{H}$.
\end{theorem}

Note that by (\ref{E:Hodge}) and (\ref{E:IbP}) we have 
\begin{equation}\label{E:divkernel}
ker\:\partial^\ast=\mathcal{H}^1.
\end{equation}

\begin{proof}
Since $\partial$ is a densely defined closed linear operator, the self-adjointness of $\partial\partial^\ast$ follows from \cite[Theorem 2, Section VII.3]{Y80} or \cite[Problem VIII.45]{RS} (note that $(\partial^\ast)^\ast=\partial$ because $\partial$ is closed).
\end{proof}

\begin{remark}\label{R:laplaceremark}\mbox{}
Our Definition \ref{D:formlaplace} is adequate if the space of $2$-forms is trivial. In the introduction we have already pointed out that several of our preceding results suggest that this is the case. 
\end{remark}

Based on Definition \ref{D:formlaplace} harmonic $1$-forms can be defined.

\begin{definition}
A $1$-form $\omega\in\mathcal{H}$ is called \emph{harmonic} if $\omega\in \dom\:\Delta_1$ and $\Delta_1\omega=0$.
\end{definition}

Obviously they form a subspace of $\mathcal{H}$. Moreover, we have the following result.

\begin{theorem}\label{T:critharmonic}
A $1$-form $\omega\in\mathcal{H}$ is harmonic if and only if it is a member of $\mathcal{H}^1(X)$. 
\end{theorem}

We will refer to $\mathcal{H}^1(X)$ as the \emph{space of harmonic $1$-forms}. 

\begin{proof}
If $\omega\in \mathcal{H}^1(X)=ker\:\partial^\ast$, then it is obviously in $\dom\:\Delta_1$ and $\Delta_1\omega=0$. Conversely, if $\omega\in\mathcal{H}$ is harmonic, then $\Delta_1\omega\in ker\:\partial^\ast$ since
\[\left\langle \partial u, \Delta_1\omega\right\rangle_\mathcal{H}=0=\left\langle u,\omega^\ast\right\rangle\]
holds for all $u\in \mathcal{F}$ if $\omega^\ast=0$. Consequently $\partial^\ast\omega=0$, that is, $\omega\in ker\:\partial^\ast$.
\end{proof}

Finally, recall the definition (\ref{E:positive}) of the bilinear map $\Gamma_\mathcal{H}$. It is the derivation $\partial\Gamma_\mathcal{H}(u)$ of $\Gamma_\mathcal{H}(u)$, $u\in\mathcal{H}$, that will be needed below. We define it as a linear functional on the possibly smaller domain
\[\dom_c\:\partial^\ast:=\left\lbrace v\in \dom\:\partial^\ast: \partial^\ast v\in C(X)\right\rbrace \subset \dom\:\partial^\ast\]
by setting
\begin{equation}\label{E:partialGamma}
\partial\Gamma_\mathcal{H}(u)(v):=-\Gamma_\mathcal{H}(u)(\partial^\ast v)=-\int_X\partial^\ast v\:d\Gamma_\mathcal{H}(u)\ \ , \ \ v\in \dom_c\:\partial^\ast.
\end{equation}
Note that for any divergence free vector field $v\in \mathcal{H}^1(X)$ we have $\partial\Gamma_\mathcal{H}(u)(v)=0$. Similarly, given an arbitrary Borel measure $\mu$ on $X$, let a functional $\partial \mu$ on $\dom_c\:\partial^\ast$ be defined by
\begin{equation}\label{E:arbitraryBorel}
\partial\mu(v):=-\int_X\partial^\ast v\:d\mu,\ \ v\in \dom_c\:\partial^\ast.
\end{equation}

\section{Navier-Stokes equations on fractals}\label{S:NS}

We finally come to an application of the results developed in Sections \ref{S:locally} and \ref{S:nontrivial}. More precisely, we will consider a Navier-Stokes type system of PDE on a compact connected topologically one-dimensional space $X$ and use Theorem \ref{T:Hcoincide} together with the decomposition (\ref{E:Hodge}) to show that just as in classical smooth cases, only steady state solutions exist for the boundary free case. In this case nontrivial solutions can exist if and only if $\check{H}^1(X)$ is nontrivial. If a boundary is specified, further nontrivial solutions may exist. For resistance forms his will be discussed in the next section. Note that these results depend only on the topological dimension, cf. Remark \ref{R:examples} below.

For an incompressible and homogeneous fluid in a Euclidean domain the Navier-Stokes system with viscosity $\nu>0$ and zero outer forcing writes
\begin{equation}\label{E:NS}
\begin{cases}
\frac{\partial u}{\partial t} +(u\cdot \nabla)u-\nu\Delta u +\nabla p=0 \\
div\:u=0.
\end{cases}
\end{equation}
$u$ is the velocity field, $p$ the pressure, and a solution consists of both. See for instance \cite{AMR, ChM, Te83}. On a smooth one-dimensional manifold the system reduces to the Euler type equation
\[\frac{\partial u}{\partial t}+\frac{\partial p}{\partial x}=0.\] 
Taking into account the classical polar decomposition $u=\nabla \varphi + w$ of vector fields into a gradient and a solenoidal part and denoting by $\mathcal{P}$ the projection onto the space of solenoidal fields, $u$ is seen to be a stationary solution (steady state solutions) because 
\[0=-\mathcal{P}(\frac{\partial p}{\partial x})=\mathcal{P}(\frac{\partial u}{\partial t})=\frac{\partial u}{\partial t}.\]

The aim of this section is to record that a similar behavior occurs for some counterpart of (\ref{E:NS}) within our setup and under the assumption that the topological dimension of $X$ is less than two. 
\begin{remark}\label{R:examples}
We emphasize that $X$ does not have to possess any smoothness properties and therefore the sequel applies in particular to highly singular spaces like finitely ramified fractals \cite{Ki01, T08}, generalized Sierpinski carpets \cite{BB99, BBKT, MTW} or Barlow-Evans-Laakso spaces \cite{BE04, La00, Stei10, S-POTA, ST-BE}, with examples of \emph{any possible Hausdorff-dimension $1\leq d<\infty$} among them. 
\end{remark}

Recall that $X$ is assumed to be a compact connected topologically one-dimensional metric space. We rewrite the divergence condition $\partial^\ast u=0$. By Theorem \ref{T:Hcoincide}, (\ref{E:Hodge}) and (\ref{E:IbP}) we have $ker\:\partial^ \ast=\mathcal{H}^1(X)$, therefore prospective solutions will be elements of $\mathcal{H}^1(X)$. We use the Hodge Laplacian $\Delta_1$ according to (\ref{E:formlaplace}) and replace the convection term by $\frac{1}{2}\partial\Gamma_\mathcal{H}(u)$. With these substitutions (\ref{E:NS}) can formally be restated as
\begin{equation}\label{E:NS2}
\begin{cases}
\frac{\partial u}{\partial t}+\frac{1}{2}\partial \Gamma_\mathcal{H}(u)-\nu\Delta_1 u+\partial p=0\\
\partial^\ast u=0.
\end{cases}
\end{equation}
This is the Navier-Stokes system on $X$ we investigate.

\begin{remark}
Of course this is just one model and already in the smooth case there is ambiguity which model to formulate (see for instance \cite{Nag97} and the references cited therein), let alone in our case. Our motivations to use formulation (\ref{E:NS2}) are as follows.

For a vector field $(u=(u_1,u_2,u_3)$ on $\mathbb{R}^3$ we have the identity
\[\frac{1}{2}\nabla|u|^2=(u\cdot \nabla)u+u\times \curl u.\]
In classical vector analysis on $\mathbb{R}^3$ usually $1$- and $2$-forms are identified with vector fields and $0$- and $3$-forms with functions. In terms of differential forms $\curl u$ is defined (respectively recovered) by viewing the vector field $u$ as a $1$-form, taking its derivation, which is a $2$-form, considering the image of this $2$-form under the Hodge star operator, which gives again a $1$-form, and then translating this $1$-form back into a vector field $\curl u$. See for instance \cite{AMR}. In the introduction we have already argued that in our setup the space of $2$-forms should be trivial, and also our Hodge Laplacian is defined with this idea in mind, cf. Remark \ref{R:laplaceremark}. If we pursue this idea, $\curl u$ should be zero, and a weak version of the remaining identity above should lead to an analog of
\begin{equation}\label{E:reducedidentity}
-\frac{1}{2}\int|u|^2\:\diverg v\:dx=\frac{1}{2}\int v\:\nabla|u|^2\:dx=\int v(u\cdot \nabla)u\:dx .
\end{equation}
In our case the Euclidean norm should be replaced by a family of norms on the fibers of $\mathcal{H}$ over $X$ as considered in \cite[Section 2]{HRT}. This suggests to replace the expression on the left hand side of (\ref{E:reducedidentity}) by
\[-\frac{1}{2}\int_X\partial^\ast vd\Gamma_\mathcal{H}(u).\]
Note also that for locally exact vector fields the integrand of the middle integral in (\ref{E:reducedidentity}) locally rewrites $\frac{1}{2}v\nabla|\nabla h|^2$. In our language this becomes $\frac{1}{2}\partial\Gamma_\mathcal{H}(\partial h)(v)$.
We finally remark that by polarization $\partial\Gamma_\mathcal{H}$ may be seen as a trilinear form on $\mathcal{H}^1(X)\times \dom_c\: A\times \dom_c\:A$, where $\dom_c\:A=\left\lbrace f\in \dom\:A: Af\in C(X)\right\rbrace$, such that if $u$ is divergence free, we have $\partial\Gamma_\mathcal{H}(u,\partial g)(\partial g)=-\partial\Gamma_\mathcal{H}(u,\partial g)(f)$ for any $f,g\in \dom_c\:A$ by the Leibniz rule (\ref{E:Leibniz}). This reminds a bit of the classical theory, see \cite[Sections I.2.3 and I.2.4]{Te83}.
\end{remark}

In the boundary free case a weak formulation of problem (\ref{E:NS2}) can be made rigorous. As usual we will test against divergence free vector fields $v\in\mathcal{H}^1=ker\:\partial^\ast$. We interpret $\partial p$ in the measure sense (\ref{E:arbitraryBorel}) and therefore obtain $\partial p(v)=0$ for all such $v$. Now recall the definitions (\ref{E:formlaplace}) and (\ref{E:partialGamma}). We say that a square integrable $\dom\:\partial^\ast$-valued function $u$ on $[0,\infty)$ provides a \emph{weak solution to (\ref{E:NS2})} with initial condition $u_0\in\mathcal{H}^1(X)=ker\:\partial^\ast$ if 
\begin{equation}\label{E:weaksol}
\begin{cases}
\left\langle u(t),v\right\rangle_\mathcal{H}-\left\langle u_0,v\right\rangle_\mathcal{H} +\int_0^t \partial\Gamma_\mathcal{H}(u(s))(v)ds+\nu\int_0^t\left\langle\partial^\ast u(s),\partial^\ast v\right\rangle_{L_2(X,m)}ds=0\\
\partial^\ast u(t)=0
\end{cases}
\end{equation}
for a.e. $t\in [0,\infty)$ and all $v\in \mathcal{H}^1(X)$. By one-dimensionality this system immediately simplifies further to
\[\begin{cases}
\left\langle u(t)-u_0,v\right\rangle_\mathcal{H}=0\\
\partial^\ast u(t)=0
\end{cases}\]
for a.e. $t\in [0,\infty)$ and all such $v$, and we obtain the following result. 

\begin{theorem}\mbox{}
Any weak solution $u$ of (\ref{E:NS2}) is harmonic and stationary, i.e. $u$ is independent of $t\in [0,\infty)$. Given an initial condition $u_0$  the corresponding weak solution is uniquely determined.
\end{theorem}

In the boundary free case the following is a consequence of Theorems \ref{T:nontrivial1} and \ref{T:nontrivial2}.
\begin{corollary}
Assume that there exists a cofinal set $\mathbb{V}$ of finite open covers $\mathcal{V}=\left\lbrace V_\beta\right\rbrace_{\beta \in I}$ satisfying (\ref{E:Lebesgue}) and such that any disconnecting set for $\mathbb{V}$ has positive capacity. Then a nontrivial solution to (\ref{E:NS2}) exists if and only if $\check{H}^1(X)$ is nontrivial.
\end{corollary}

We conclude the section with a heuristic remark concerning the pressure $p$.
\begin{remark}
We did not define strong solutions to (\ref{E:NS}), i.e. solutions that do not need testing. However, if there were a differentiable $dom\:\partial^\ast$-valued function $u$ such that (\ref{E:NS}) holds in a measure-valued sense, then this $u$ should also satisfy the weak formulation (\ref{E:NS2}). Therefore any such $u$ must be stationary, and by (\ref{E:partialGamma}) and (\ref{E:arbitraryBorel}) we would obtain 
\[p=-\frac{1}{2}\Gamma_\mathcal{H}(u),\]
seen as an equality of measures.
\end{remark}

\section{The case of resistance forms}\label{S:resistance}

In this section we additionally assume that the local Dirichlet form under consideration is induced by a regular resistance form $(\mathcal{E},\overline{\mathcal{F}})$ on $X$. See \cite{IRT, Ki03, Ki12, T08} for background and precise definitions. In the resistance form context Neumann derivatives are well-defined, and it is not difficult to see that if the Navier-Stokes system (\ref{E:NS2}) is considered with a nonempty boundary, it may have additional nontrivial solutions arising from solutions of a related Neumann problem.

Let $(\mathcal{E},\overline{\mathcal{F}})$ be a local resistance form on $X$, \cite[Definition 2.8]{Ki03} and let $R$ be the associated resistance metric. We consider the topological space $(X,R)$. An open ball of radius $r>0$ and with center $x\in X$ in this space is denoted by $B_R(x,r)$. For any Borel regular measure $m$ on $(X,R)$ such that $0<m(B(x,r))<\infty$, the space $(\overline{F}\cap L_2(X,m), \mathcal{E}_1)$ is Hilbert, and denoting by $\mathcal{F}$ the closure of $C_0(X)\cap\overline{\mathcal{F}}$ in it, we obtain a local regular Dirichlet form $(\mathcal{E},\mathcal{F})$ on $L_2(X,m)$. See for instance \cite[Section 9]{Ki12}. We assume the following:
\begin{assumption}\label{A:resistassumption}
\text{The space $(X,R)$ is compact, connected, and topologically one-dimensional.}
\end{assumption}
Under Assumption \ref{A:resistassumption} the results of the previous sections may be applied to the induced Dirichlet form $(\mathcal{E},\mathcal{F})$.

\begin{remark}
It is conjectured that any set that carries a regular resistance form becomes a topologically one-dimensional space when equipped with the associated resistance metric.
\end{remark}

We provide a few notions and references related to resistance forms. Let $B\subset X$ be a finite set. The points of $B$ will be interpreted as boundary points. By $G_B$ we denote the Green operator associated with the boundary $B$ with respect to $(\mathcal{E},\overline{\mathcal{F}})$, \cite[Definition 5.6]{Ki03}, and $\mathcal{D}_{B,0}^L$ its image in $\overline{\mathcal{F}}$. Let $\mathcal{H}_B$ denote the $B$-harmonic functions with respect to $(\mathcal{E},\overline{\mathcal{F}})$, \cite[Definition 2.16]{Ki03}, and note that $\overline{\mathcal{F}}=\overline{F}_B\oplus\mathcal{H}_B$, where
\[\overline{\mathcal{F}}_B:=\left\lbrace u\in\overline{\mathcal{F}}: u|_B=0\right\rbrace.\]
A $B$-harmonic function $h$ is harmonic on $B^c$ in the Dirichlet form sense, more precisely, it satisfies
\[\mathcal{E}(h,\psi)=0\]
for all $\psi\in\overline{\mathcal{F}}_B$. The space $\mathcal{D}^L:=\mathcal{D}^L_{B,0}+\mathcal{H}_B$ is seen to be independent of the choice of $B$, \cite[Theorem 5.10]{Ki03}. For any $u\in \mathcal{D}^L$ and any $p\in X$ the Neumann derivative $(du)_p$ of $u$ at $p$ can be defined. We refer the reader to \cite[Theorems 6.6 and 6.8]{Ki03}. Now let $\varphi$ be a function on $B$. A function $h_\varphi\in\overline{\mathcal{F}}$ is called a \emph{solution to the Neumann problem on $B^c$ with boundary values $\varphi$} if it 
is harmonic on $B^c$ and satisfies 
\[(dh)_p=\varphi(p)\]
for all $p\in B$. 
Such a Neumann solution $h_\varphi$ exists if and only if $\varphi$ is such that $$\sum_{p\in B}\varphi(p)=0.$$ If it exists, it is unique in $\mathcal{H}_B$ up to an additive constant. 
These last two assertions follow from the connectedness of $X$ in effective resistance metric assumed above. This connectedness is equivalent to the fact that only constant functions have zero energy. Therefore the linear map $\{h(p)\}_{p\in B}\mapsto\{(dh)_p\}_{p\in B}$ is linear with constant kernel, which implies the assertion by elementary linear algebra.

Regarding a nonempty finite set $B\subset X$ as a boundary, the system (\ref{E:NS2}) is to be viewed on $B^c$
(playing the role of a domain). Recalling (\ref{E:weaksol}), we will therefore use test vector fields $v$ vanishing outside $B^c$. Given an open set $U\subset X$, let
\[\mathcal{H}(U):=\clos\lin\left\lbrace a\otimes\mathbf{1}_V: a\in\mathcal{B},\ V\subset U \text{ open }\right\rbrace.\] 
For simplicity we will regard $\mathcal{B}_b(U)$ as subspace of $\mathcal{B}_b(X)$ under continuation by zero. We will use the following result.
\begin{remark}\mbox{}
\begin{enumerate}
\item[(i)] For any open set $U\subset X$ the equality
\[\mathcal{H}(U)=\clos\lin\left\lbrace \omega\mathbf{1}_U: \omega\in\mathcal{H}\right\rbrace\]
holds: Clearly $\mathcal{H}(U)$ is contained in the space on the right-hand side. To see the converse inclusion, note that the space on the right-hand side contains all simple tensors $a\otimes b\in\mathcal{B}\otimes\mathcal{B}_b(U)$, which can be seen using a Dynkin type argument as in \cite[Theorem 4.1]{H11}. Therefore it must also contain all limits of linear combinations of such elements, and by their denseness in $\mathcal{H}$ the desired inclusion follows using (\ref{E:boundedactions}).
\item[(ii)] For a finite set $B\subset X$ the space $\mathcal{H}(B^c)$ is contained in the space $Im\:\partial$: For any $a\otimes \mathbf{1}_{B^c} \in\mathcal{B}\otimes\mathcal{B}_b(B^c)$ the continuous function $a$ clearly is constant on each of the finitely many distinct points of $B$. Therefore $a\otimes\mathbf{1}_{B^c}$ is exact by Proposition \ref{P:gradient}, and by (i) all members of $\mathcal{H}(B^c)$ are.
\end{enumerate}
\end{remark}

Let $B\subset X$ be finite. We will now say that a square integrable $dom\:\partial^\ast$-valued function $u$ on $[0,T]$ provides a \emph{weak solution to (\ref{E:NS2}) on $B^c$} if 
\begin{equation}\label{E:NSboundary}
\begin{cases}
\left\langle u(t),v\right\rangle_\mathcal{H}-\left\langle u(0),v\right\rangle_\mathcal{H} +\int_0^t \partial\Gamma_\mathcal{H}(u(s))(v)ds+\int_0^t\left\langle \partial^\ast u(s),\partial^\ast v\right\rangle_{L_2(X,m)}ds=0\\
\left\langle u(t), \partial\psi\right\rangle_\mathcal{H}=0
\end{cases}
\end{equation}
for a.e. $t\in [0,T]$, all $v\in \dom\:\partial^\ast \cap \mathcal{H}(B^c)$ and all $\psi\in\overline{\mathcal{F}}_B$.

\begin{theorem}
Let Assumption $\ref{A:resistassumption}$ be valid and let $B\subset X$ be finite. If $h$ is the unique, up to an additive constant, harmonic function on  $B^c$ with normal derivatives $\varphi$ on $B$, then 
\[u(t)=\partial h,\ \  t\in[0,\infty),\]
is the unique weak solution to (\ref{E:NS2}) on $B^c$ with the Neumann boundary values $\varphi$ on $B$.
\end{theorem}

\begin{proof} For $u\equiv\partial h$ the right hand side of the first equation vanishes. The divergence condition holds by harmonicity, 
\[\left\langle \partial h,\partial\psi\right\rangle_\mathcal{H}=\mathcal{E}(h,\psi)=0\]
for any $\psi\in \overline{\mathcal{F}}_B$. 
\end{proof}

\begin{remark}
Similarly as before, if $h$ were a 'strong' solution to (\ref{E:NS2}) on $B^c$, then we should observe the identity
\[p(t)=-\frac{1}{2}\Gamma(h), \ \  t\in[0,\infty).\]
\end{remark}

\end{document}